\documentclass[11pt]{amsart}
\usepackage{amsmath,amssymb}
\usepackage[margin=1.15in]{geometry}
\usepackage[shortlabels]{enumitem}
\usepackage[colorlinks=true,linkcolor=blue,citecolor=blue,urlcolor=blue]{hyperref}
\hypersetup{pdftitle={Mellin transforms of powers of Moebius transformations and complete elliptic integrals},
  pdfauthor={K. Srinivasa Raghava}}

\newtheorem{theorem}{Theorem}[section]
\newtheorem{lemma}[theorem]{Lemma}
\newtheorem{proposition}[theorem]{Proposition}
\newtheorem{corollary}[theorem]{Corollary}
\theoremstyle{definition}
\newtheorem{example}[theorem]{Example}
\theoremstyle{remark}
\newtheorem{remark}[theorem]{Remark}
\numberwithin{equation}{section}

\DeclareMathOperator{\Ti}{Ti}
\DeclareMathOperator{\artanh}{artanh}
\DeclareMathOperator{\gd}{gd}
\DeclareMathOperator{\Log}{Log}
\newcommand{\pv}{\mathrm{p.v.}\!\!}
\newcommand{\F}{{}_2F_1}
\newcommand{\dm}{\frac{dt}{t^2-1}}
\newcommand{\one}{\mathbf{1}}

\begin{document}

\title[Mellin transforms of powers of M\"obius transformations]{Mellin transforms of powers of M\"obius transformations\\ and complete elliptic integrals}
\author{K. Srinivasa Raghava}
\address{Pie Mathematics Association}
\email{srinivasaraghavak@gmail.com}

\begin{abstract}
We study Mellin integrals of powers of two reciprocal M\"obius maps, one positive
on the integration interval and one with an interior zero and pole. For the
positive map, multiplying the transform by $\sin(\pi a)$ gives a function
symmetric in the Mellin weight $a$ and the exponent. An expansion in
Mittag-Leffler polynomials yields a transformation
extending Landen's identity. At half weight the mixed second logarithmic
moment of two singular maps depends only on the smaller pole. For a reciprocal
quadratic ratio we give a recurrence with an explicit error bound for every
even half-weight moment. At a nonzero reduced rational Mellin weight $p/q$ in $(-1,1)$,
the function field of the generating series for integer exponents has generic
degree $q^2$ and genus $(q-1)(q-2)/2$.
We determine both cubic periods on the real interval between the singular parameters
and evaluate quartic sampling periods in gamma values at a point of complex
multiplication.
\end{abstract}

\maketitle

\section{Introduction}

For \(w,r>1\) consider the reciprocal M\"obius maps
\begin{equation}\label{eq:maps}
 \tau_w(t)=\frac{1+wt}{t+w},\qquad \sigma_r(t)=\frac{rt-1}{t-r}.
\end{equation}
Both satisfy \(f(1/t)=1/f(t)\). The first increases from \(w^{-1}\) to
\(w\) on the positive axis. The second has a zero at \(r^{-1}\), a pole at \(r\),
and a negative sign between them. Define
\begin{equation}\label{eq:GH}
 G_w(a,s)=\int_0^\infty t^a\{\tau_w(t)^s-1\}\dm,\qquad
 H_r(a,s)=\int_0^\infty t^a\{|\sigma_r(t)|^s-1\}\dm .
\end{equation}
For \(|\Re a|<1\) the first integral converges for every \(s\), and the
second for \(|\Re s|<1\). Differentiating at \(s=0\) gives logarithmic
moments. The apparent pole at \(t=1\) is removable in both integrals.

The coordinate \(x=(t-1)/(t+1)\) turns \(\tau_w\) into a dilation and
\(\sigma_r\) into an inversion followed by a dilation. The differential
\(2\,dt/(t^2-1)=dx/x\) is preserved up to sign. Two consequences are
especially simple to state. The Mellin weight and the exponent can be
exchanged in \(\sin(\pi a)G_w(a,s)\) on their common strip
(Theorem~\ref{thm:symmetry}). At half weight, even the locations of
some logarithmic singularities can vary without changing the integral.
For example, Corollary~\ref{cor:min} gives
\[
 \int_0^\infty\sqrt t\,\log|\sigma_3(t)|\,\log|\sigma_r(t)|\dm
   =\frac{\pi^3}{6}\qquad(r\ge3).
\]
The underlying formula depends only on the smaller of the two parameters.
In the even and odd parts of \(G_w(\tfrac12,s)\) the same changes of variable
lead to a transformation extending Landen's identity and to difference
equations involving Legendre functions.

The quadratic ratio
\begin{equation}\label{eq:factor}
 R_{c,b}(t)=\frac{ct^2+bt+1}{t^2+bt+c}
           =\tau_u(t)\tau_v(t),\qquad u+v=b,\quad uv=c,
\end{equation}
connects these transforms to algebraic curves. For \(c>1\) and
\(b>-2\sqrt c\), the mixed moments of the factors give the logarithmic
moments of their positive product. Sampling its
deformation at positive integer exponents gives a generating function of
a different kind. At reduced Mellin weight \(a=p/q\), Theorem~\ref{geo:degreegenus}
determines its field over \(\mathbb C(z)\). The generic genus is
\((q-1)(q-2)/2\). For \(q=3\) and \(0<(b/(c+1))^2<1\) we compute a real
and an imaginary period forming a basis. For \(q=4\) we identify the
three elliptic quotients and determine
which carry the sampling differential. At \((c,b)=(5,-2)\) the ratio in
the \(x\)-coordinate is the Cayley kernel
\((2x^2+2x+1)/(2x^2-2x+1)\), and all three quotients have complex multiplication.

On \(a+s=1\) the hypergeometric evaluation of \(G_w\) gives the generalized
elliptic integrals of Anderson, Qiu, Vamanamurthy and Vuorinen \cite{AQVV}.
The weights \(1/2,1/3,1/4,1/6\) connect it with the classical elliptic
integral and Ramanujan's alternative signatures \cite{BBG,Ram1914}.
The polynomial expansion uses the Mittag-Leffler polynomials
\cite{Bateman,ML,Roman}. We give the differential and cycle normalizations
when using classical elliptic period evaluations.

\paragraph{Notation.}
Write \(\kappa=(w-1)/(w+1)\), \(\lambda=(r-1)/(r+1)\), and \(h=\log w\).
We use principal logarithms and roots unless branches are continued
explicitly. Factors \(1/\sin(\pi a)\) or \(\cot(\pi a)\) at \(a=0\)
are interpreted by a limit whenever one exists. We write \(\F={}_2F_1\),
\(P_\nu(z)=\F(-\nu,\nu+1;1;(1-z)/2)\), and use the elliptic
\emph{modulus} \(k\):
\[
 K(k)=\tfrac{\pi}{2}\F(\tfrac12,\tfrac12;1;k^2),\qquad
 E(k)=\tfrac{\pi}{2}\F(-\tfrac12,\tfrac12;1;k^2).
\]
The complementary integrals \(K',E'\) have modulus \(\sqrt{1-k^2}\).
We use Legendre's chi function and the inverse tangent integral in the form
\[
 \chi_2(z)=\sum_{j\ge0}\frac{z^{2j+1}}{(2j+1)^2}\quad(|z|\le1),
 \qquad \Ti_2(v)=\frac{\chi_2(iv)}{i}.
\]
We write \(\beta(2)=\Ti_2(1)\) for Catalan's constant,
\(\varphi=(1+\sqrt5)/2\) for the golden ratio and
\(\gd\rho=2\arctan(\tanh(\rho/2))\) for the Gudermannian.

\section{Preliminaries}\label{sec:prelim}

Put $x=(t-1)/(t+1)$, so that $t=(1+x)/(1-x)$ and $-1<x<1$. Then
\begin{equation}\label{eq:xcoord}
\frac{2\,dt}{t^2-1}=\frac{dx}{x},\qquad t^a=\Big(\frac{1+x}{1-x}\Big)^{a},\qquad
\tau_w(t)=\frac{1+\kappa x}{1-\kappa x},\qquad \sigma_r(t)=-\frac{\lambda+x}{\lambda-x},
\end{equation}
where $\kappa=(w-1)/(w+1)$ and $\lambda=(r-1)/(r+1)$. The formula for $\tau_w$ holds for every
$w>0$, with $\kappa\in(-1,1)$. In particular
\begin{equation}\label{eq:logtau}
\log\tau_w(t)=2\artanh(\kappa x),\qquad \log|\sigma_r(t)|=\log\Big|\frac{\lambda+x}{\lambda-x}\Big| .
\end{equation}

\begin{lemma}\label{lem:basic}
Let $w>0$ and $r>1$.
\begin{enumerate}[(i)]
\item If $u=\tau_w(t)$, then $(u-1)/(u+1)=\kappa x$. If $u=\sigma_r(t)$, then
$(u-1)/(u+1)=\lambda/x$. So in the coordinate $x$ the map $\tau_w$ is the dilation
$x\mapsto\kappa x$, and the map $\sigma_r$ is $x\mapsto\lambda/x$.
\item If $w\ne1$ and $u=\tau_w(t)$, then $\dfrac{du}{u^2-1}=\dfrac{dt}{t^2-1}$. If $u=\sigma_r(t)$, then
$\dfrac{du}{u^2-1}=-\dfrac{dt}{t^2-1}$.
\item The map $\sigma_r$ is an involution. The map $t\mapsto-\sigma_r(t)$ sends $(1/r,r)$
increasingly onto $(0,\infty)$, and its inverse is $\tau_r$.
\end{enumerate}
\end{lemma}

\begin{proof}
Equation~\eqref{eq:xcoord} gives (i); the logarithmic differential $dx/x$ then gives (ii).
The reciprocal map $x\mapsto\lambda/x$ squares to the identity, whereas
$-\sigma_r$ acts by $x\mapsto x/\lambda$ on $(-\lambda,\lambda)$ and has inverse $\tau_r$.
\end{proof}

\begin{remark}\label{rem:commute}
In the $x$ coordinate, reciprocity is $x\mapsto-x$. A fractional linear map
$(\alpha x+\beta)/(\gamma x+\delta)$ commutes with it exactly when
$\alpha\gamma=\beta\delta=0$. The resulting dilations and scaled inversions
are $\tau_w$ and $\sigma_r$ for complex $w,r\ne\pm1$, including the limits
$\tau_\infty(t)=t$ and $\sigma_\infty(t)=-t$.
\end{remark}

With $t=e^{2y}$ and $w=e^{2\rho}$, $r=e^{2\rho}$ we have $\kappa=\lambda=\tanh\rho$ and
\begin{equation}\label{eq:hyp}
\tau_w(t)=\frac{\cosh(y+\rho)}{\cosh(y-\rho)},\qquad
\sigma_r(t)=\frac{\sinh(y+\rho)}{\sinh(y-\rho)},\qquad \frac{dt}{t^2-1}=\frac{dy}{\sinh2y}.
\end{equation}

\begin{lemma}\label{lem:conv}
Let $w>0$ and $r>1$.
\begin{enumerate}[(i)]
\item The integral $G_w(a,s)$ converges absolutely for $|\Re a|<1$ and $s\in\mathbb{C}$, and it
is an analytic function of $(a,s)$ in this domain. Moreover $G_1=0$ and
$G_{1/w}(a,s)=G_w(a,-s)$.
\item The integral $H_r(a,s)$ converges absolutely for $|\Re a|<1$ and $|\Re s|<1$, and it is
an analytic function of $(a,s)$ in this domain.
\item $G_w(-a,-s)=-G_w(a,s)$ and $H_r(-a,-s)=-H_r(a,s)$.
\end{enumerate}
\end{lemma}

\begin{proof}
(i) The power series of $\artanh$ shows that $|\artanh(\kappa x)|\le|x|\artanh|\kappa|$ for
$|x|\le1$, so \eqref{eq:logtau} gives $|\log\tau_w(t)|\le|h|\,|x|$ with
$h=\log w=2\artanh\kappa$. Since
$|e^z-1|\le|z|e^{|z|}$ and $|x|/|t^2-1|=(t+1)^{-2}$, the integrand of $G_w$ is bounded by
$|sh|\,e^{|sh|}\,t^{\Re a}(t+1)^{-2}$. This bound is integrable for $|\Re a|<1$ and locally
uniform in $(a,s)$, which proves the first statement. The other two follow from $\tau_1=1$
and $\tau_{1/w}=1/\tau_w$.

(ii) At $t=1$, $\log|\sigma_r(t)|$ is smooth and vanishes. At $0,\infty$ the
integrand of $H_r$ is respectively $O(t^{\Re a})$ and $O(t^{\Re a-2})$.
At $t_0=r^{\pm1}$ it is $O(|t-t_0|^{-|\Re s|})$. These bounds give integrable
majorants on compact subsets of the two strips.

(iii) Under $t\mapsto1/t$ the three factors $t^a,\tau_w(t),|\sigma_r(t)|$
are inverted, while the integral against $dt/(t^2-1)$ changes sign.
\end{proof}
\begin{lemma}\label{lem:trig}
Let $F(t)=f(x)$ in \eqref{eq:xcoord} and suppose the integrals converge absolutely. Then
\[
\int_0^\infty\sqrt t\,F(t)\dm=
\begin{cases}
 \displaystyle\int_0^{\pi/2} f(\sin u)\,\dfrac{du}{\sin u},&f\ \text{odd},\\[4pt]
 \displaystyle\int_0^{\pi/2} f(\sin u)\,du,&f\ \text{even}.
\end{cases}
\]
\end{lemma}

\begin{proof}
Since $\sqrt t=(1+x)/\sqrt{1-x^2}$, the integral equals
$\tfrac12\int_{-1}^1(1+1/x)f(x)\,dx/\sqrt{1-x^2}$.
Pair $x$ with $-x$ and put $x=\sin u$.
\end{proof}
\subsection{Two elementary integrals}

\begin{lemma}\label{lem:beta}
Let $\zeta_1,\zeta_2\in\mathbb{C}\setminus(-\infty,0]$ and $|\Re a|<1$. Then
\[
\int_0^\infty\frac{t^a\,dt}{(t+\zeta_1)(t+\zeta_2)}=\frac{\pi}{\sin\pi a}\cdot\frac{\zeta_1^a-\zeta_2^a}{\zeta_1-\zeta_2},
\]
where the right side is understood as a limit if $a=0$ or $\zeta_1=\zeta_2$.
\end{lemma}

\begin{proof}
For $-1<\Re a<0$, the beta integral
$\int_0^\infty t^a/(t+\zeta)\,dt=-\pi\zeta^a/\sin\pi a$
\cite[Section~1.1]{AAR} extends holomorphically from $\zeta>0$ to the slit plane.
Take the divided difference at $\zeta_1,\zeta_2$ and continue in $a$.
\end{proof}

\begin{lemma}\label{lem:tan}
For $|\Re a|<1$,
\[
\int_0^\infty\frac{t^a-1}{t^2-1}\,dt=\frac\pi2\tan\frac{\pi a}{2},\qquad
\pv\int_0^\infty\frac{t^a}{t^2-1}\,dt=\frac\pi2\tan\frac{\pi a}{2}.
\]
\end{lemma}

\begin{proof}
Both sides of the first formula are analytic in $a$, so let $a$ be real. The substitution
$t\mapsto1/t$ maps $(1,\infty)$ onto $(0,1)$, and it shows that the first integral equals
\[
\int_0^1\frac{t^{-a}-t^{a}}{1-t^2}\,dt=\sum_{k\ge0}\Big\{\frac{1}{2k+1-a}-\frac{1}{2k+1+a}\Big\}
=\frac\pi2\tan\frac{\pi a}{2}.
\]
The geometric expansion has terms of constant sign for real $a$.
The cotangent partial fractions \cite[(4.22.3)]{DLMF} give the sum.
Finally $\pv\int_0^\infty dt/(t^2-1)=0$ by inversion.
\end{proof}

\section{The hypergeometric formula and the symmetry}\label{sec:symmetry}

\begin{lemma}\label{lem:dw}
For $w>0$, $|\Re a|<1$ and $s\in\mathbb{C}$,
\begin{align}
\frac{\partial G_w}{\partial w}(a,s)&=s\int_0^\infty t^a(1+wt)^{s-1}(t+w)^{-s-1}\,dt\notag\\
&=\frac{\pi a s}{\sin\pi a}\,w^{-a-s-2}\,\F\big(1+a,1+s;2;1-w^{-2}\big).\label{eq:dw}
\end{align}
\end{lemma}

\begin{proof}
Differentiation cancels $t^2-1$, since $\partial_w\tau_w=(t^2-1)/(t+w)^2$.
The resulting integrand is $O(t^{\Re a}(1+t)^{-2})$ locally uniformly in $w$.
With $u=wt/(1+wt)$ and $Z=1-w^{-2}$ its integral is
\[
s\,w^{-a-s-2}\int_0^1 u^a(1-u)^{-a}(1-Zu)^{-s-1}\,du .
\]
Euler's integral \cite[(15.6.1)]{DLMF} and
$\Gamma(1+a)\Gamma(1-a)=\pi a/\sin\pi a$ give \eqref{eq:dw}.
\end{proof}

\begin{theorem}\label{thm:symmetry}
Let $w>1$, $|\Re a|<1$ and $s\in\mathbb{C}$. Then
\begin{equation}\label{eq:Gint}
G_w(a,s)=\frac{\pi a s}{2\sin\pi a}\int_0^{1-w^{-2}}(1-z)^{\frac{a+s-1}{2}}\,\F(1+a,1+s;2;z)\,dz .
\end{equation}
Consequently, for $|\Re a|<1$ and $|\Re s|<1$,
\begin{equation}\label{eq:symmetry}
\sin(\pi a)\,G_w(a,s)=\sin(\pi s)\,G_w(s,a).
\end{equation}
\end{theorem}

\begin{proof}
Since $G_1=0$, we have $G_w=\int_1^w\partial_vG_v\,dv$. Insert \eqref{eq:dw} and put
$z=1-v^{-2}$, so that $v^{-a-s-2}\,dv=\frac12(1-z)^{(a+s-1)/2}\,dz$. This gives
\eqref{eq:Gint}. The right side of \eqref{eq:Gint}, multiplied by $\sin\pi a$, is symmetric in
$a$ and $s$, because $\F(1+a,1+s;2;z)$ is.
\end{proof}

\begin{lemma}\label{lem:twostep}
For $w>1$, $|\Re a|<1$ and $s\in\mathbb{C}$,
\begin{equation}\label{eq:twostep}
G_w(a,s+2)-G_w(a,s)=\frac{\pi a}{\sin\pi a}\,(w^2-1)\,w^{-a-s-3}\,\F\big(1+a,2+s;2;1-w^{-2}\big).
\end{equation}
\end{lemma}

\begin{proof}
Since $\tau_w^2-1=(w^2-1)(t^2-1)/(t+w)^2$, the left side equals
$(w^2-1)\int_0^\infty t^a(1+wt)^{s}(t+w)^{-s-2}\,dt$. The substitution used in the proof of
Lemma~\ref{lem:dw} turns this integral into
$w^{-a-s-3}\,\Gamma(1+a)\Gamma(1-a)\,\F(2+s,1+a;2;1-w^{-2})$.
\end{proof}

\subsection{The expansion in Mittag-Leffler polynomials}
Define polynomials $g_n$ by
\begin{equation}\label{eq:ML}
\Big(\frac{1+z}{1-z}\Big)^{s}=\sum_{n\ge0}g_n(s)\,z^n\qquad(|z|<1).
\end{equation}
The binomial expansion gives
\[
 g_n(s)=\sum_{j=1}^n2^j\binom{n-1}{j-1}\binom{s}{j}\quad(n\ge1),
 \qquad g_n(-s)=(-1)^n g_n(s).
\]
In the notation of \cite[Section~4.1.6]{Roman}, $n!\,g_n$ are the polynomials
of Mittag-Leffler \cite{ML}; see also Bateman \cite{Bateman}.
\begin{theorem}\label{thm:bilinear}
For $w>1$, $|\Re a|<1$ and $s\in\mathbb{C}$,
\begin{equation}\label{eq:bilinear}
G_w(a,s)=\frac{\pi}{2\sin\pi a}\sum_{n\ge1}g_n(a)\,g_n(s)\,\kappa^n,\qquad \kappa=\frac{w-1}{w+1}.
\end{equation}
The series $\sum_{n\ge1}g_n(a)g_n(s)\kappa^n$ converges absolutely and
defines an entire function of $(a,s)$. Thus $\sin(\pi a)G_w(a,s)$ extends
to an entire function of both variables.
\end{theorem}

\begin{proof}
First we show that, for $n\ge1$ and $|\Re a|<1$,
\begin{equation}\label{eq:MLmoment}
\int_0^\infty t^a\,\frac{(t-1)^{n-1}}{(t+1)^{n+1}}\,dt=\frac{\pi}{2\sin\pi a}\,g_n(a).
\end{equation}
Multiply the left side by $z^{n-1}$ with $|z|<1$ and sum over $n\ge1$. Since
$|(t-1)/(t+1)|\le1$, the sum may be taken inside the integral, and with $\zeta=(1+z)/(1-z)$
the result is
\[
\frac{1}{1-z}\int_0^\infty\frac{t^a\,dt}{(t+1)(t+\zeta)}
=\frac{1}{1-z}\cdot\frac{\pi}{\sin\pi a}\cdot\frac{\zeta^a-1}{\zeta-1}
=\frac{\pi}{2\sin\pi a}\cdot\frac{\zeta^a-1}{z}
\]
by Lemma~\ref{lem:beta}; note that $\Re\zeta>0$. Comparing the coefficients of $z^{n-1}$ gives
\eqref{eq:MLmoment}.

By \eqref{eq:xcoord} and \eqref{eq:ML},
\[
\frac{\tau_w(t)^s-1}{t^2-1}=\sum_{n\ge1}g_n(s)\,\kappa^n\,\frac{(t-1)^{n-1}}{(t+1)^{n+1}} .
\]
For any $\varepsilon>0$, Cauchy's estimate in \eqref{eq:ML} gives
$|g_n(s)|\le C(1+\varepsilon)^n$ uniformly on compact $s$-sets.
Choose $(1+\varepsilon)^2\kappa<1$. Since
$|t-1|^{n-1}(t+1)^{-n-1}\le(t+1)^{-2}$, termwise integration gives
\eqref{eq:bilinear} by \eqref{eq:MLmoment}. Applying the same estimate in
$a$ gives local uniform convergence on $\mathbb C^2$.
\end{proof}

It follows that $G_w$ continues meromorphically in $a$, with at most
simple poles at the nonzero integers.
\begin{corollary}\label{cor:integer}
Let $w>1$. For $|\Re a|<1$,
\[
G_w(a,1)=\frac{\pi}{\sin\pi a}\,(w^a-1),\qquad
G_w(a,2)=\frac{\pi a}{\sin\pi a}\,(w^2-1)\,w^{a-1},
\]
and for $|z|<1/w$
\begin{equation}\label{eq:genfn}
\sum_{m\ge1}G_w(a,m)\,z^m=\frac{\pi z}{(1-z^2)\sin\pi a}\Big\{\Big(\frac{w-z}{1-wz}\Big)^{a}-1\Big\}.
\end{equation}
For every positive integer $m$, the product $\sin(\pi a)\,G_w(a,m)$ is a polynomial in $a$,
$w$, $w^{-1}$ and $w^a$.
\end{corollary}

\begin{proof}
For $n\ge1$, $g_n(1)=2$ and $g_n(2)=4n$. Substitute these in
\eqref{eq:bilinear} and use $\sum_{n\ge1}g_n(a)\kappa^n=w^a-1$.
More generally $g_n(m)$ is a polynomial in $n$ of degree $m-1$ for
$n\ge1$ by the binomial expansion above. Replace $n$ in that polynomial by $\kappa\partial_\kappa
=\tfrac12(w^2-1)\partial_w$ and apply it to $w^a-1$ to obtain the polynomial assertion.
For $|z|<1/w$, the bound $|\tau_w^m-1|\le m w^{m-1}|\tau_w-1|$ and
$(\tau_w-1)/(t^2-1)=(w-1)/((t+1)(t+w))$ justify summation under the integral.
With $\zeta=(w-z)/(1-wz)$ one finds
\[
\sum_{m\ge1}\big\{\tau_w(t)^m-1\big\}z^m=\frac{z(w-1)(t-1)}{(1-z)(1-wz)(t+\zeta)},
\]
and Lemma~\ref{lem:beta} with $\zeta_1=\zeta$, $\zeta_2=1$ finishes the proof, since
$\zeta-1=(w-1)(1+z)/(1-wz)$.
\end{proof}

\begin{remark}\label{rem:watson}
Let $|\Re a|<1$ and $|\Re(a+s)|<1$. As $w\to\infty$ we have $\tau_w(t)\to t$, and
$\tau_w(t)$ lies between $1$ and $t$, so
$|\tau_w(t)^s-1|\le|s\log t|\max(1,t^{\Re s})$. By dominated convergence and
Lemma~\ref{lem:tan},
\[
\lim_{w\to\infty}G_w(a,s)=\int_0^\infty\frac{t^{a+s}-t^a}{t^2-1}\,dt
=\frac{\pi}{2}\Big\{\tan\frac{\pi(a+s)}{2}-\tan\frac{\pi a}{2}\Big\}.
\]
Comparing with \eqref{eq:Gint} we obtain
\[
\int_0^1(1-z)^{\frac{a+s-1}{2}}\,\F(1+a,1+s;2;z)\,dz
=\frac{2\sin\frac{\pi a}{2}\sin\frac{\pi s}{2}}{a\,s\,\cos\frac{\pi(a+s)}{2}} .
\]
At $as=0$ the quotients have their continuous limits. The beta-integrated
absolute terms are $O(n^{(\Re(a+s)-3)/2})$ unless the series terminates.
Termwise integration thus identifies the result with Watson's
summation theorem \cite[Section~3.3]{Bailey}, \cite[(16.4.6)]{DLMF} for
${}_3F_2\big(1+a,1+s,1;\frac{a+s+3}{2},2;1\big)$. So \eqref{eq:Gint} contains this case of Watson's
theorem as the limit $w\to\infty$.
\end{remark}

\begin{example}\label{ex:half}
At $a=s=\tfrac12$, \eqref{eq:ML} gives
$g_{2m}(\tfrac12)=g_{2m+1}(\tfrac12)=4^{-m}\binom{2m}{m}$.
The power series of $K$ therefore sums \eqref{eq:bilinear} to
$G_w(\tfrac12,\tfrac12)=(1+\kappa)K(\kappa)-\pi/2$.
\end{example}

\section{The relation between \texorpdfstring{$H_r$}{H\_r} and \texorpdfstring{$G_r$}{G\_r}}\label{sec:relation}

\begin{theorem}\label{thm:relation}
For $r>1$, $|\Re a|<1$ and $|\Re s|<1$,
\begin{equation}\label{eq:relation}
H_r(a,s)+\cos(\pi a)\,G_r(a,s)=\tan\frac{\pi s}{2}\,\big\{\pi+\sin(\pi a)\,G_r(a,s)\big\}.
\end{equation}
Equivalently,
\begin{equation}\label{eq:relation2}
 H_r(a,s)=\pi\tan\frac{\pi s}{2}
  -\frac{\cos\pi(a+s/2)}{\cos(\pi s/2)}\,G_r(a,s).
\end{equation}
\end{theorem}

\begin{proof}
First take real $|a|,|s|<1$, $s\ne0$, and set
$f(z)=z^a(\sigma_r(z)^s-1)/(z^2-1)$. Use principal logarithms in the upper
half-plane. There $\Im\sigma_r(z)=(1-r^2)\Im z/|z-r|^2<0$.
The positive-axis boundary value of $\sigma_r^s$ is $|\sigma_r|^s$
outside $(1/r,r)$ and $e^{-i\pi s}|\sigma_r|^s$ inside. At $z=-t+i0$,
$\sigma_r(-t)=\tau_r(t)$ and $z^a=e^{i\pi a}t^a$.

Integrate over the upper half-disk with clockwise indentations at
$0,1/r,1,r$. The bounds $O(|z|^a)$ at zero,
$O(|z-r^{\pm1}|^{-|s|})$ at the zero and pole, and $O(|z|^{a-2})$ at
infinity make the corresponding arc integrals vanish. At $-1$ the
singularity is removable. At $1$ the upper branch has residue
$(e^{-i\pi s}-1)/2$. Thus, with $J=\pv\int_{1/r}^r t^a|\sigma_r(t)|^s\dm$
and the principal value taken at $1$, we obtain
\begin{equation}\label{eq:contour1}
 H_r(a,s)+e^{i\pi a}G_r(a,s)+(e^{-i\pi s}-1)J
   =\tfrac{i\pi}{2}(e^{-i\pi s}-1).
\end{equation}
Multiply by $e^{i\pi s/2}$ and take real parts. The term containing $J$
is purely imaginary and the result is \eqref{eq:relation2}.
The substitution $u=-\sigma_r(t)$ in $J$ has inverse $t=\tau_r(u)$
and preserves the principal value at $1$, since its derivative there is
nonzero. Lemma~\ref{lem:tan} now gives
$J=G_r(s,a)+(\pi/2)\tan(\pi s/2)$, so the imaginary part of
\eqref{eq:contour1} also yields \eqref{eq:symmetry}.
Continuity gives $s=0$; holomorphy from Lemma~\ref{lem:conv}, first in $a$
and then in $s$, extends the formula throughout the stated strips.
\end{proof}
\begin{corollary}\label{cor:indep}
Let $r>1$ and $w>1$.
\begin{enumerate}[(i)]
\item For $0<a<1$: \ $H_r(a,1-2a)=\pi\cot\pi a$.
\item For $0<s<1$: \ $H_r(1-s,s)-G_r(1-s,s)=\pi\tan\frac{\pi s}{2}$.
\item For $|s|<1$: \ $H_r(-s,s)+G_r(-s,s)=\pi\tan\frac{\pi s}{2}$.
\item $\displaystyle\int_0^\infty\sqrt{t\,\tau_w(t)}\,\log\frac{t}{\tau_w(t)}\dm
=\int_0^\infty\sqrt{t\,|\sigma_r(t)|}\,\log\frac{|\sigma_r(t)|}{t}\dm=\frac{\pi^2}{2}.$
\end{enumerate}
\end{corollary}

\begin{proof}
For (i)--(iii) the coefficient of $G_r$ in \eqref{eq:relation2} is
respectively $0,+1,-1$ on the indicated lines.
For (iv) let $\phi(a)=G_w(a,1-a)$ for $0<a<1$. By \eqref{eq:symmetry} we have
$\phi(a)=\phi(1-a)$, so $\phi'(\frac12)=0$. Differentiating under the integral sign,
\[
\phi'(a)=\int_0^\infty\Big\{t^a\tau_w^{1-a}\log\frac{t}{\tau_w}-t^a\log t\Big\}\dm .
\]
Moreover $\int_0^\infty\sqrt t\,\log t\,\frac{dt}{t^2-1}=\frac{\pi^2}{2}$, as one sees by
differentiating the first formula of Lemma~\ref{lem:tan} at $a=\frac12$. This gives the first integral in (iv). By (ii),
$\psi(a)=H_r(a,1-a)=\pi\cot\frac{\pi a}{2}+G_r(a,1-a)$, so $\psi'(\frac12)=-\pi^2$, and the
same computation gives the second integral.
\end{proof}

In particular the radical integrals
\[
\int_0^\infty t^{1/4}\big(\sqrt{|\sigma_r(t)|}-1\big)\dm
=\int_0^\infty \sqrt t\,\big(\sqrt{|\sigma_r(t)|}-\sqrt{\tau_r(t)}\big)\dm=\pi
\]
follow by taking $a=\tfrac14$ in (i) and $s=\tfrac12$ in (ii).
Next we compare the moments of $\log\tau_r$ and of $\log|\sigma_r|$. Write
\[
m_n(a)=\int_0^\infty t^a\log^n\tau_r(t)\dm,\qquad \mu_n(a)=\int_0^\infty t^a\log^n|\sigma_r(t)|\dm .
\]
Let $T_1=1$, $T_3=2$, $T_5=16,\dots$ be the tangent numbers, defined by
$\tan x=\sum_{k\ \mathrm{odd}}T_kx^k/k!$, and put $T_k=0$ for even $k$.

\begin{corollary}\label{cor:momentrelation}
For $|\Re a|<1$ and $n\ge1$,
\[
\mu_n+\cos(\pi a)\,m_n=\pi\Big(\frac\pi2\Big)^{n}T_n
+\sin(\pi a)\sum_{\substack{k\ \mathrm{odd}\\ k<n}}\binom nk\Big(\frac\pi2\Big)^{k}T_k\,m_{n-k}.
\]
In particular $\mu_3+\cos(\pi a)m_3=\pi^4/4+(3\pi/2)\sin(\pi a)m_2$.
\end{corollary}

\begin{proof}
By Lemma~\ref{lem:conv}, $m_n$ and $\mu_n$ are the $n$th derivatives in $s$
at zero of $G_r$ and $H_r$. Compare coefficients in \eqref{eq:relation}.
\end{proof}
\begin{remark}\label{rem:zero}
The derivative at $a=0$ of the generating function \eqref{eq:ML} is
$2\artanh z$, so at zero weight
\eqref{eq:bilinear} and \eqref{eq:relation} become
\[
G_w(0,s)=\sum_{n\ \mathrm{odd}}\frac{g_n(s)}{n}\,\kappa^n,\qquad
H_r(0,s)=\pi\tan\frac{\pi s}{2}-G_r(0,s).
\]
By Lemma~\ref{lem:conv}(iii) both functions are odd in $s$, so for $a=0$ all moments of
even order vanish. The moments of order one are
\[
\int_0^\infty\log\tau_w(t)\dm=2\chi_2(\kappa),\qquad
\int_0^\infty\log|\sigma_r(t)|\dm=\frac{\pi^2}{2}-2\chi_2(\lambda).
\]
\end{remark}

\section{The line \texorpdfstring{$a+s=1$}{a+s=1} and complete elliptic integrals}\label{sec:legendre}

\begin{theorem}\label{thm:legendre}
Let $w>1$ and $|\Re a|<1$, $a\ne0$. Then
\begin{equation}\label{eq:legendre}
G_w(a,1-a)=\frac{\pi}{2\sin\pi a}\Big\{\F\big(a,1-a;1;1-w^{-2}\big)-1\Big\}.
\end{equation}
If $r>1$ and $0<\Re a<1$, then
\begin{equation}\label{eq:legendreH}
H_r(a,1-a)=\frac{\pi}{2\sin\pi a}\Big\{\F\big(a,1-a;1;1-r^{-2}\big)+1+2\cos\pi a\Big\}.
\end{equation}
\end{theorem}

\begin{proof}
When $a+s=1$ the factor $(1-z)^{(a+s-1)/2}$ in \eqref{eq:Gint} equals $1$, and
$a(1-a)\,\F(1+a,2-a;2;z)=\frac{d}{dz}\F(a,1-a;1;z)$ by \cite[(15.5.1)]{DLMF}. This gives
\eqref{eq:legendre}. By Corollary~\ref{cor:indep}(ii) and analytic continuation in $a$,
$H_r(a,1-a)=\pi\cot\frac{\pi a}{2}+G_r(a,1-a)$, and
$\pi\cot\frac{\pi a}{2}=\pi(1+\cos\pi a)/\sin\pi a$. This gives \eqref{eq:legendreH}.
\end{proof}

The hypergeometric term is $P_{-a}(2w^{-2}-1)$. For $0<a<1$ it is also
$2\mathcal K_a(\sqrt{1-w^{-2}})/\pi$, where
$\mathcal K_a(k)=\tfrac\pi2\F(a,1-a;1;k^2)$ is the generalized elliptic integral \cite{AQVV}.
The cases $a=1/3,1/4,1/6$ give Ramanujan's signatures $3,4,6$ \cite{BerndtV,BBG,Cooper}.

\begin{corollary}\label{cor:bases}
For $w>1$, $Z=1-w^{-2}$ and $\kappa=(w-1)/(w+1)$,
\[
 G_w(\tfrac12,\tfrac12)=K(\sqrt Z)-\tfrac\pi2,\qquad
 G_w(\tfrac14,\tfrac34)=\sqrt{2(1+\kappa)}\,K(\sqrt\kappa)-\tfrac\pi{\sqrt2}.
\]
\end{corollary}
\begin{proof}
The first identity is \eqref{eq:legendre} at $a=1/2$. In the second use
$Z=4\kappa/(1+\kappa)^2$ and the quadratic transformation
$\F(\frac14,\frac34;1;Z)=\sqrt{1+\kappa}\,\F(\frac12,\frac12;1;\kappa)$
\cite[(15.8.15)]{DLMF}.
\end{proof}
\begin{corollary}\label{cor:gamma}
For $0<a<1$,
\begin{equation}\label{eq:gamma}
G_{\sqrt2}(a,1-a)=\frac{\pi}{2\sin\pi a}\Big\{\frac{\sqrt\pi}{\Gamma(\frac{1+a}{2})\,\Gamma(1-\frac a2)}-1\Big\}.
\end{equation}
In particular
\begin{align*}
G_{\sqrt2}\big(\tfrac12,\tfrac12\big)&=\frac{\Gamma(\frac14)^2}{4\sqrt\pi}-\frac\pi2,
&G_{\sqrt2}\big(\tfrac13,\tfrac23\big)&=\frac{\sqrt3\,\Gamma(\frac13)^3}{2^{7/3}\pi}-\frac{\pi}{\sqrt3},\\
G_{\sqrt2}\big(\tfrac14,\tfrac34\big)&=\frac{\Gamma(\frac18)\Gamma(\frac38)}{4\sqrt\pi}-\frac{\pi}{\sqrt2},
&G_{\sqrt2}\big(\tfrac16,\tfrac56\big)&=\frac{3^{1/4}\,\Gamma(\frac14)^2}{2\sqrt{2\pi}}-\pi .
\end{align*}
\end{corollary}

\begin{proof}
At $w=\sqrt2$ the argument in \eqref{eq:legendre} is $\frac12$, and Gauss's second summation
theorem \cite[(15.4.28)]{DLMF} gives
$\F(a,1-a;1;\frac12)=\sqrt\pi/\{\Gamma(\frac{1+a}{2})\Gamma(1-\frac a2)\}$. The four special
cases follow from the reflection, duplication and triplication formulas for the Gamma
function \cite[Sections~1.2 and~1.5]{AAR}.
\end{proof}

\begin{example}\label{ex:CM}
Corollaries~\ref{cor:bases} and \ref{cor:gamma} recover the classical
values at the first three singular moduli \cite{BB,SC,Zucker}. They give
$K(1/\sqrt2)=\Gamma(\frac14)^2/(4\sqrt\pi)$ and, at $w=\sqrt2$ where
$\sqrt\kappa=\sqrt2-1$,
\[
 K(\sqrt2-1)=
 \frac{\sqrt{\sqrt2+1}\,\Gamma(\frac18)\Gamma(\frac38)}{2^{13/4}\sqrt\pi}.
\]
For the third modulus use Ramanujan's transformation
\cite[Theorem~5.6]{BBG}, valid for $0\le p<1$,
\begin{equation}\label{eq:BBG}
\F\Big(\tfrac13,\tfrac23;1;\frac{27p^2(1+p)^2}{4(1+p+p^2)^3}\Big)
=\frac{1+p+p^2}{\sqrt{1+2p}}\,
 \F\Big(\tfrac12,\tfrac12;1;\frac{p^3(2+p)}{1+2p}\Big).
\end{equation}
With $p=(\sqrt3-1)/2$ the arguments are $1/2$ and $k_3^2$, where
$k_3=(\sqrt6-\sqrt2)/4$. Equation~\eqref{eq:gamma} now gives
$K(k_3)=3^{1/4}\Gamma(\frac13)^3/(2^{7/3}\pi)$.
With $k_2=\sqrt2-1$, the singular-modulus ratios are
$K'(k_j)/K(k_j)=\sqrt j$ for $j=2,3$ \cite{BB,Zucker}. Hence
$G_{1/k_j}(\tfrac12,\tfrac12)=\sqrt j\,K(k_j)-\pi/2$, at
$1/k_2=\sqrt2+1$ and $1/k_3=\sqrt6+\sqrt2$. The complementary parameter
$w=\sqrt6-\sqrt2$ gives $G_w(\tfrac12,\tfrac12)=K(k_3)-\pi/2$.

Conversely, choosing $p=(\sqrt{3+2\sqrt3}-1)/2$ makes the second
argument $1/2$. With $w^2=10+6\sqrt3$, \eqref{eq:legendre} becomes
\[
G_w(\tfrac13,\tfrac23)=
 \frac{(\sqrt3+1)\Gamma(\frac14)^2}{4\sqrt\pi(3+2\sqrt3)^{1/4}}-\frac{\pi}{\sqrt3}.
\]
Shpot \cite{Shpot} records the hypergeometric values at this argument
$1-w^{-2}$ and its complement. In signature $4$, $\kappa=1/2$ at $w=3$,
so $G_3(\tfrac14,\tfrac34)=\sqrt3\,\Gamma(\frac14)^2/(4\sqrt\pi)-\pi/\sqrt2$.
\end{example}

\subsection{Half-integer values of \texorpdfstring{$s$}{s} for \texorpdfstring{$a=\frac12$}{a=1/2}}
\begin{theorem}\label{thm:elliptic}
Let $w>1$, $k=\sqrt{1-w^{-2}}$, $K=K(k)$ and $E=E(k)$. Then
\begin{align*}
G_w\big(\tfrac12,\tfrac12\big)&=K-\frac\pi2, &G_w\big(\tfrac12,-\tfrac12\big)&=\frac Kw-\frac\pi2,\\
G_w\big(\tfrac12,\tfrac32\big)&=2wE-\frac Kw-\frac\pi2, &G_w\big(\tfrac12,-\tfrac32\big)&=2E-K-\frac\pi2,
\end{align*}
and
\[
H_w\big(\tfrac12,\tfrac12\big)=K+\frac\pi2,\qquad H_w\big(\tfrac12,-\tfrac12\big)=-\frac Kw-\frac\pi2 .
\]
For every integer $n$ there are rational functions $A_n,B_n\in\mathbb{Q}(w)$ such that
$G_w(\frac12,n+\frac12)=A_nK+B_nE-\frac\pi2$.
\end{theorem}

\begin{proof}
Put $x=\cos\theta$ in \eqref{eq:xcoord}. Then
$\tau_w=(1+\kappa\cos\theta)/(1-\kappa\cos\theta)$, $\sqrt t=(1+\cos\theta)/\sin\theta$ and
$dx/x=-\tan\theta\,d\theta$, so that for every $s\in\mathbb{C}$
\begin{equation}\label{eq:trigG}
G_w\big(\tfrac12,s\big)=\frac12\int_0^\pi\big\{\tau_w^s-1\big\}\big(1+\sec\theta\big)\,d\theta ,
\qquad \tau_w=\frac{1+\kappa\cos\theta}{1-\kappa\cos\theta}.
\end{equation}
For $s=\pm\frac12$ we have
$\tau_w^{\pm1/2}=(1\pm\kappa\cos\theta)/\sqrt{1-\kappa^2\cos^2\theta}$. The
terms that are odd under $\theta\mapsto\pi-\theta$ integrate to zero, and what remains is
\[
G_w\big(\tfrac12,\pm\tfrac12\big)=\frac{1\pm\kappa}{2}\int_0^\pi\frac{d\theta}{\sqrt{1-\kappa^2\cos^2\theta}}-\frac\pi2
=(1\pm\kappa)K(\kappa)-\frac\pi2 .
\]
The positive-sign identity and Corollary~\ref{cor:bases} give
$K(k)=(1+\kappa)K(\kappa)$, where $k=2\sqrt\kappa/(1+\kappa)$.
The negative-sign identity follows in the stated form since
$(1-\kappa)/(1+\kappa)=1/w$.

For the other values of $s$ we use Lemma~\ref{lem:twostep} with $a=\frac12$: for every
integer $n$,
\begin{equation}\label{eq:step}
G_w\big(\tfrac12,n+\tfrac32\big)-G_w\big(\tfrac12,n-\tfrac12\big)
=\frac\pi2\,(w^2-1)\,w^{-n-3}\,\F\big(\tfrac32,n+\tfrac32;2;k^2\big).
\end{equation}
By \cite[(15.5.1)]{DLMF} and the formulas for the derivatives of $K$ and $E$
\cite[(19.4.1), (19.4.2)]{DLMF},
\[
\F\big(\tfrac32,\tfrac32;2;k^2\big)=\frac{4}{\pi}\cdot\frac{E-(1-k^2)K}{k^2(1-k^2)},\qquad
\F\big(\tfrac32,\tfrac12;2;k^2\big)=\frac{4}{\pi}\cdot\frac{K-E}{k^2},
\]
and \eqref{eq:step} with $n=0$ and $n=-1$ gives the second line. For general $n$ the functions
$f_n=\F(\frac32,n+\frac32;2;k^2)$ satisfy the contiguous relation
\cite[Section~2.5]{AAR}, \cite[Section~15.5(ii)]{DLMF}
\[
\big(\tfrac12-n\big)f_{n-1}+\big(2n+1-nk^2\big)f_n+\big(n+\tfrac32\big)(k^2-1)f_{n+1}=0,
\]
in which the coefficients of $f_{n-1}$ and $f_{n+1}$ never vanish. So every $f_n$ is a
combination of $f_0$ and $f_{-1}$ with coefficients in $\mathbb{Q}(k^2)\subset\mathbb{Q}(w)$,
and the last statement follows from \eqref{eq:step} by induction in both directions.

The values of $H_w$ follow from \eqref{eq:relation} with $a=\frac12$, which reads
$H_w(\frac12,s)=\tan\frac{\pi s}{2}\{\pi+G_w(\frac12,s)\}$.
\end{proof}

For example, setting $w=1/k$ with $0<k<1$ gives
\begin{align*}
 \int_0^\infty\sqrt t\left(\sqrt{\frac{t+k}{1+kt}}-1\right)\dm
  &=K'(k)-\frac{\pi}{2},\\
 \int_0^\infty\sqrt t\left(\sqrt{\left|\frac{t-k}{1-kt}\right|}-1\right)\dm
  &=K'(k)+\frac{\pi}{2}.
\end{align*}
\begin{corollary}\label{cor:legendrerel}
Let $w,\tilde w>1$ with $w^{-2}+\tilde w^{-2}=1$. Then
\[
\Big\{G_w\big(\tfrac12,\tfrac12\big)+\frac\pi2\Big\}\Big\{G_{\tilde w}\big(\tfrac12,-\tfrac32\big)+\frac\pi2\Big\}
+\Big\{G_{\tilde w}\big(\tfrac12,\tfrac12\big)+\frac\pi2\Big\}\Big\{G_{w}\big(\tfrac12,-\tfrac32\big)+\frac\pi2\Big\}=\pi .
\]
In particular, at $w=\tilde w=\sqrt2$,
\[
\Big\{G_{\sqrt2}\big(\tfrac12,\tfrac12\big)+\frac\pi2\Big\}\Big\{G_{\sqrt2}\big(\tfrac12,-\tfrac32\big)+\frac\pi2\Big\}=\frac\pi2 .
\]
\end{corollary}

\begin{proof}
The moduli $k=\sqrt{1-w^{-2}}$ and $\tilde k=\sqrt{1-\tilde w^{-2}}$ are complementary. By
Theorem~\ref{thm:elliptic} the left side is
$K(k)\{2E(\tilde k)-K(\tilde k)\}+K(\tilde k)\{2E(k)-K(k)\}$, which equals $\pi$ by Legendre's
relation $E(k)K(\tilde k)+E(\tilde k)K(k)-K(k)K(\tilde k)=\frac\pi2$
\cite[(19.7.1)]{DLMF}.
\end{proof}

\section{Moments of the logarithms}\label{sec:moments}

\subsection{Mixed moments for positive factors}
In this section $w_j=e^{h_j}$. For $n\ge1$ and $|\Re a|<1$ let
\[
N_n(\xi)=\begin{cases}\sinh^{n-2}\xi\,\sinh a\xi,& n\ \text{odd},\\[2pt] \sinh^{n-2}\xi\,\cosh a\xi,& n\ \text{even},\end{cases}
\qquad
D_n(\xi_1,\dots,\xi_n)=\sum_{j=1}^{n}\frac{N_n(\xi_j)}{\prod_{k\ne j}(\cosh\xi_j-\cosh\xi_k)} .
\]
Since $N_n$ is even and analytic for $|\Im\xi|<\pi$, it is locally an analytic function
of $\cosh\xi$ for $\xi\ge0$, including at zero where $\xi^2$ is analytic in $\cosh\xi-1$.
Thus $D_n$ is a divided difference and extends continuously to repeated nodes \cite{deBoor}. In particular,
\begin{equation}\label{eq:D2}
D_1(\xi)=\frac{\sinh a\xi}{\sinh\xi},\qquad
D_2(\xi_1,\xi_2)=\frac{\cosh a\xi_1-\cosh a\xi_2}{\cosh\xi_1-\cosh\xi_2}
=\frac{\sinh\frac{a(\xi_1+\xi_2)}{2}\,\sinh\frac{a(\xi_1-\xi_2)}{2}}{\sinh\frac{\xi_1+\xi_2}{2}\,\sinh\frac{\xi_1-\xi_2}{2}} .
\end{equation}

\begin{theorem}\label{thm:moments}
For $n\ge1$, $|\Re a|<1$ and $h_1,\dots,h_n\ge0$,
\begin{equation}\label{eq:moments}
\int_0^\infty t^a\prod_{j=1}^{n}\log\tau_{w_j}(t)\dm
=\frac{\pi}{\sin\pi a}\int_0^{h_1}\!\!\cdots\!\int_0^{h_n}D_n(\xi_1,\dots,\xi_n)\,d\xi_1\cdots d\xi_n .
\end{equation}
\end{theorem}

\begin{proof}
Write $\Lambda_n$ for the moment on the left and set $t=e^\xi$, $c_j=\cosh h_j$.
The derivative $\partial_{h_j}\log\tau_{w_j}=\sinh\xi/(\cosh\xi+c_j)$ is bounded by
$|\tanh(\xi/2)|$. Lemma~\ref{lem:conv} gives the same bound times $h_j$ for the logarithm.
On a compact parameter set take $A=\max|\Re a|<1$. All mixed
derivatives under consideration are dominated by a constant times
$e^{A|\xi|}/\cosh^2(\xi/2)$, which is integrable. Thus
\begin{equation}\label{eq:mixedderiv}
\frac{\partial^n\Lambda_n}{\partial h_1\cdots\partial h_n}
=\frac12\int_{-\infty}^{\infty}\frac{e^{a\xi}\sinh^{n-1}\xi}{\prod_{j}(\cosh\xi+c_j)}\,d\xi .
\end{equation}
We need two elementary integrals. For $h,h'>0$,
\begin{equation}\label{eq:el1}
\frac12\int_{-\infty}^{\infty}\frac{e^{a\xi}\,d\xi}{\cosh\xi+\cosh h}=\frac{\pi}{\sin\pi a}\cdot\frac{\sinh ah}{\sinh h},
\end{equation}
\begin{equation}\label{eq:el2}
\frac12\int_{-\infty}^{\infty}e^{a\xi}\sinh\xi\,\Big\{\frac{1}{\cosh\xi+\cosh h}-\frac{1}{\cosh\xi+\cosh h'}\Big\}\,d\xi
=-\frac{\pi}{\sin\pi a}\,\big(\cosh ah-\cosh ah'\big).
\end{equation}
Equation \eqref{eq:el1} follows from Lemma~\ref{lem:beta} with roots $e^{\pm h}$.
Differentiating the left side of \eqref{eq:el2} in $h$ and integrating by parts in
$\xi$ gives $-a\sinh h$ times the left side of \eqref{eq:el1}. Both sides of
\eqref{eq:el2} vanish at $h=h'$.

For distinct positive $h_j$ put $u=\cosh\xi$ and $Q_j=\prod_{k\ne j}(c_j-c_k)$.
When $n$ is odd the partial fraction coefficient of $(u+c_j)^{-1}$ in
$(u^2-1)^{(n-1)/2}/\prod_k(u+c_k)$ is $\sinh^{n-1}h_j/Q_j$.
For even $n$ remove one factor $\sinh\xi$ from the numerator of \eqref{eq:mixedderiv}.
The resulting coefficients are $-\sinh^{n-2}h_j/Q_j$ and sum to zero by the
degree at infinity. Hence one may subtract $(u+c_n)^{-1}$ from every simple fraction
in the even case. Equations \eqref{eq:el1} and \eqref{eq:el2} now give the mixed
derivative $\pi D_n/\sin\pi a$ in the odd and even cases, respectively.
Continuity includes coincident and zero nodes. Integrating from zero in each
$h_j$ gives \eqref{eq:moments}, since $\Lambda_n$ vanishes on each coordinate face.
\end{proof}

\begin{corollary}\label{cor:cov}
For $w,w_1,w_2\ge1$ and $|\Re a|<1$,
\begin{align}
\int_0^\infty t^a\log\tau_w(t)\dm&=\frac{\pi}{\sin\pi a}\int_0^{\log w}\frac{\sinh a\xi}{\sinh\xi}\,d\xi,\label{eq:first}\\
\int_0^\infty t^a\log\tau_{w_1}(t)\log\tau_{w_2}(t)\dm
&=\frac{\pi}{\sin\pi a}\int_0^{\log w_1}\!\!\int_0^{\log w_2}\frac{\cosh a\xi_1-\cosh a\xi_2}{\cosh\xi_1-\cosh\xi_2}\,d\xi_1\,d\xi_2 .\label{eq:cov}
\end{align}
\end{corollary}

\begin{proposition}\label{prop:complex}
Let $w,w_1,w_2\in\mathbb{C}\setminus(-\infty,0]$ and $|\Re a|<1$. Put $\kappa=(w-1)/(w+1)$,
define $\log\tau_w(t)=2\artanh(\kappa x)$ as in \eqref{eq:logtau}, and let $\log w$ denote the
principal logarithm. Then \eqref{eq:first} and \eqref{eq:cov} hold, with the integrals on the
right taken along straight segments.
\end{proposition}

\begin{proof}
The map $h\mapsto\tanh(h/2)$ takes $|\Im h|<\pi$ onto
$\Omega=\mathbb C\setminus((-\infty,-1]\cup[1,\infty))$. Since $\Omega$ is star-shaped about
zero, $\kappa x$ stays in $\Omega$ for $|x|\le1$, and $2\artanh(\kappa x)$ is the principal
logarithm of $\tau_w$. Locally uniformly in $h$, it is bounded by $C|x|$,
which gives holomorphy of both moment integrals as in Lemma~\ref{lem:conv}.
The two kernels in \eqref{eq:D2} are holomorphic on the corresponding strips,
with removable coincident-node values. Analytic continuation from positive $h_j$
in each variable proves the assertion.
\end{proof}

\begin{example}\label{ex:arctan}
Take $w=e^{i\phi}$ with $0<\phi<\pi$. Then $\kappa=i\tan\frac\phi2$, so that
$\log\tau_w(t)=2i\arctan\big(x\tan\frac\phi2\big)$ with $x=(t-1)/(t+1)$, and
Proposition~\ref{prop:complex} gives, for $|\Re a|<1$,
\[
\int_0^\infty t^a\arctan\Big(\tan\frac\phi2\cdot\frac{t-1}{t+1}\Big)\dm
=\frac{\pi}{2\sin\pi a}\int_0^{\phi}\frac{\sin a\theta}{\sin\theta}\,d\theta .
\]
For $\phi=\pi/2$ the arctangent is $\arctan t-\pi/4$. At $a=0,1/3,1/2$ the
three integrals are respectively $\beta(2)$, $(\pi/2)\log 2$, and
$(\pi/2)\log(1+\sqrt2)$. At zero, integration by parts gives
$\int_0^{\pi/2}\theta\,d\theta/\sin\theta
=-2\int_0^{\pi/4}\log\tan v\,dv=2\beta(2)$.
For the other two use the triple- and double-angle formulas on the right.
\end{example}

\begin{proposition}\label{prop:rational}
For integers $0<p<n$ let $J_{p/n}(w)$ denote the first moment
in \eqref{eq:first} at $a=p/n$. For $w\in\mathbb C\setminus(-\infty,0]$,
\begin{equation}\label{eq:rational-logs}
 J_{p/n}(w)=\frac{\pi}{2\sin(\pi p/n)}
 \sum_{\substack{\zeta^{2n}=1\\\zeta\ne\pm1}}
 \zeta^{-n}(\zeta^p-\zeta^{-p})
 \Log\!\left(\frac{w^{1/n}-\zeta}{1-\zeta}\right).
\end{equation}
The root and logarithm are principal. For real $w=e^h>1$ this becomes
\[
J_{p/n}(w)=\frac{2\pi}{\sin(\pi p/n)}\sum_{j=1}^{n-1}(-1)^{j+1}\sin\frac{\pi jp}{n}\,
\arctan\Big(\cot\frac{\pi j}{2n}\,\tanh\frac{h}{2n}\Big).
\]
\end{proposition}

\begin{proof}
Substitute $z=e^{\xi/n}$ in \eqref{eq:first}. Apart from $\pi n/\sin(\pi p/n)$
the resulting differential is $z^{n-p-1}(z^{2p}-1)\,dz/(z^{2n}-1)$.
It has no polynomial part, and the poles at $\pm1$ cancel. At
$\zeta_j=e^{i\pi j/n}$ its residue is $(-1)^ji\sin(\pi jp/n)/n$.
The path from $1$ to $w^{1/n}$ lies in $|\arg z|<\pi/n$. This convex
sector excludes every pole $\zeta$, so its image under
$z\mapsto(z-\zeta)/(1-\zeta)$ is convex, contains $1$, and omits $0$.
It therefore avoids the negative axis, giving the principal logarithms in
\eqref{eq:rational-logs}. For $w>1$, pairing conjugate poles reduces their logarithms to the change of
$\arg(z-\zeta_j)$ from $1$ to $e^{h/n}$, for $1\le j<n$.
This change lies in $(0,\pi)$ and its tangent is
$\cot(\pi j/(2n))\tanh(h/(2n))>0$, which fixes the principal arctangent.
\end{proof}

For $p=1$ and $n=2,3,4,6$ the sum simplifies. With $h=\log w$,
\begin{align*}
\int_0^\infty t^{1/2}\log\tau_w(t)\dm&=\pi\arcsin\frac{w-1}{w+1},\\
\int_0^\infty t^{1/3}\log\tau_w(t)\dm&=2\pi\Big\{\arctan\frac{2w^{2/3}+1}{\sqrt3}-\frac\pi3\Big\},\\
\int_0^\infty t^{1/4}\log\tau_w(t)\dm&=\sqrt2\,\pi\Big\{\sqrt2\arctan\big(\sqrt2\sinh\tfrac h4\big)-\arctan\big(\sinh\tfrac h4\big)\Big\},\\
\int_0^\infty t^{1/6}\log\tau_w(t)\dm&=2\pi\Big\{\arctan\big(\sinh\tfrac h6\big)+\arctan\big(2\sinh\tfrac h6\big)-\sqrt3\arctan\frac{2\sinh\frac h6}{\sqrt3}\Big\}.
\end{align*}
At $w=(1+\sqrt3)^{3/2}$ the third-weight integral equals
$2\pi(5\pi/12-\pi/3)=\pi^2/6$.

\subsection{Mixed moments for singular factors}\label{sec:sigma}
Let $r_j>1$ and $h_j=\log r_j$. We write $\ell_j(t)=\log|\sigma_{r_j}(t)|$, and $\one_j$ for
the indicator function of the interval $I_j=(1/r_j,r_j)$.

\begin{lemma}\label{lem:product}
For $n\ge1$, real $a$ with $|a|<1$, and $r_1,\dots,r_n>1$,
\begin{equation}\label{eq:product}
\pv\int_0^\infty t^a\prod_{j=1}^{n}\big\{\ell_j(t)-i\pi\one_j(t)\big\}\dm
+e^{i\pi a}\int_0^\infty t^a\prod_{j=1}^{n}\log\tau_{r_j}(t)\dm=\frac{\pi^{n+1}}{2}\,(-i)^{n-1}.
\end{equation}
The principal value refers to $t=1$, and it is needed only for the term
$\prod_j(-i\pi\one_j)$.
\end{lemma}

\begin{proof}
Use the indented contour of Theorem~\ref{thm:relation} with
$f(z)=z^a\prod_j\Log\sigma_{r_j}(z)/(z^2-1)$. The positive-axis logarithms have boundary
values $\ell_j-i\pi\one_j$ and the negative-axis values are $\log\tau_{r_j}$.
The logarithmic singularities at $r_j^{\pm1}$ give vanishing arc integrals,
as do the endpoints by $|a|<1$. The pole at $-1$ cancels and the indentation at
$1$ contributes $-i\pi$ times the residue $(-i\pi)^n/2$.
\end{proof}

\begin{theorem}\label{thm:sigmamoments}
Let $r,r_1,r_2>1$, $h=\log r$, $h_j=\log r_j$ and $|\Re a|<1$. Then
\begin{align}
\int_0^\infty t^a\log|\sigma_r(t)|\dm&=\frac{\pi^2}{2}-\pi\cot(\pi a)\int_0^{h}\frac{\sinh a\xi}{\sinh\xi}\,d\xi,\label{eq:sfirst}\\
\int_0^\infty t^a\log|\sigma_{r_1}(t)|\log|\sigma_{r_2}(t)|\dm
&=\pi^2\!\int_0^{\min(h_1,h_2)}\!\frac{\sinh a\xi}{\sinh\xi}\,d\xi\notag\\
&\qquad-\pi\cot(\pi a)\int_0^{h_1}\!\!\int_0^{h_2}\!D_2(\xi_1,\xi_2)\,d\xi_1\,d\xi_2 .\label{eq:scov}
\end{align}
\end{theorem}

\begin{proof}
For real $a$, the real part of \eqref{eq:product} at $n=1$ gives \eqref{eq:sfirst}
by \eqref{eq:first}. At $n=2$ it gives the singular moment plus
$\cos(\pi a)$ times the regular mixed moment as
$\pi^2\pv\int_{I_1\cap I_2}t^a\dm$. The intersection is $(1/r,r)$ for
$r=\min(r_1,r_2)$. Putting $t=e^\xi$ and pairing $\xi$ with $-\xi$ evaluates
the principal value as $\int_0^{\log r}\sinh(a\xi)/\sinh\xi\,d\xi$.
Apply \eqref{eq:cov} and continue analytically in $a$.
\end{proof}

The imaginary part of \eqref{eq:product} at $n=2$ also gives
\[
\int_{1/r_2}^{r_2}t^a\log|\sigma_{r_1}(t)|\dm+\int_{1/r_1}^{r_1}t^a\log|\sigma_{r_2}(t)|\dm
=\frac{\pi^2}{2}+\int_0^{h_1}\!\!\int_0^{h_2}D_2(\xi_1,\xi_2)\,d\xi_1\,d\xi_2 .
\]

\begin{corollary}\label{cor:min}
For $r_1,r_2>1$,
\begin{equation}\label{eq:min}
\int_0^\infty\sqrt t\,\log|\sigma_{r_1}(t)|\,\log|\sigma_{r_2}(t)|\dm=\pi^2\arcsin\frac{r-1}{r+1},\qquad r=\min(r_1,r_2).
\end{equation}
Consequently, if $1<r_1<r_2\le r_3<r_4$, then
\[
\int_0^\infty\sqrt t\,\log\Big|\frac{\sigma_{r_2}(t)}{\sigma_{r_1}(t)}\Big|\,\log\Big|\frac{\sigma_{r_4}(t)}{\sigma_{r_3}(t)}\Big|\dm=0 .
\]
\end{corollary}

\begin{proof}
At $a=1/2$ the cotangent vanishes and
$\int_0^h d\xi/(2\cosh(\xi/2))=\arcsin\tanh(h/2)$. Expand the product for the second identity.
\end{proof}

With $L_\alpha(u)=\log|(\sin\alpha+\sin u)/(\sin\alpha-\sin u)|$,
Lemma~\ref{lem:trig} rewrites \eqref{eq:min} as
\begin{equation}\label{eq:g-min}
 \int_0^{\pi/2}L_\alpha(u)L_\beta(u)\,du=\pi^2\alpha,\qquad 0<\alpha\le\beta<\pi/2 .
\end{equation}

For $r>1$ put $h=\log r$ and use the moments $m_n,\mu_n$ of
Corollary~\ref{cor:momentrelation}. Inserting \eqref{eq:first} and \eqref{eq:cov}
in its cases $n=2,3$ gives
\[
\begin{aligned}
 \mu_2+\cos(\pi a)m_2&=\pi^2\int_0^hD_1(\xi)\,d\xi,\\
 \mu_3+\cos(\pi a)m_3&=\frac{\pi^4}{4}+
    \frac{3\pi^2}{2}\int_0^h\!\int_0^hD_2(\xi_1,\xi_2)\,d\xi_1d\xi_2 .
\end{aligned}
\]
\begin{example}\label{ex:third}
At $a=\frac13$ the first identity above and the case $n=3$ of
Proposition~\ref{prop:rational} give, for every $r>1$,
\[
\int_0^\infty t^{1/3}\big\{2\log^2|\sigma_r(t)|+\log^2\tau_r(t)\big\}\dm
=2\sqrt3\,\pi^2\Big\{\arctan\frac{2r^{2/3}+1}{\sqrt3}-\frac\pi3\Big\}.
\]
At $r=(1+\sqrt3)^{3/2}$ the value is $\pi^3/(2\sqrt3)$.
\end{example}

\begin{remark}\label{rem:fourier}
For purely imaginary $a=i\omega$, $\omega>0$, the formulas \eqref{eq:first} and
\eqref{eq:sfirst} become Fourier cosine transforms. In the notation of \eqref{eq:hyp},
\begin{align*}
\int_0^\infty\cos(2\omega y)\,\log\frac{\cosh(y+\rho)}{\cosh(y-\rho)}\,\frac{dy}{\sinh2y}
&=\frac{\pi}{\sinh\pi\omega}\int_0^{\rho}\frac{\sin2\omega\xi}{\sinh2\xi}\,d\xi,\\
\int_0^\infty\cos(2\omega y)\,\log\Big|\frac{\sinh(y+\rho)}{\sinh(y-\rho)}\Big|\,\frac{dy}{\sinh2y}
&=\frac{\pi^2}{4}-\pi\coth(\pi\omega)\int_0^{\rho}\frac{\sin2\omega\xi}{\sinh2\xi}\,d\xi .
\end{align*}
\end{remark}

\section{Fourier series and the case \texorpdfstring{$a=\frac12$}{a=1/2}}\label{sec:half}

For $x=\cos\theta$ and $F_j(t)=f_j(x)$ with $f_j$ odd, Lemma~\ref{lem:trig}
applied to their product gives, whenever the integrals converge absolutely,
\begin{equation}\label{eq:halfint}
\int_0^\infty\sqrt t\,F_1(t)\,F_2(t)\dm=\frac12\int_0^\pi f_1(\cos\theta)\,f_2(\cos\theta)\,d\theta .
\end{equation}
\subsection{A Parseval formula}
For a complex number $q$ with $|q|\le1$ let
\begin{equation}\label{eq:Lq}
\mathcal{L}_q(\theta)=4\sum_{j\ \mathrm{odd}}\frac{q^j}{j}\cos j\theta\qquad(0<\theta<\pi).
\end{equation}
The series for $\mathcal{L}_q$ converges in $L^2(0,\pi)$. For $|q|<1$ it converges
absolutely, and its sum is $2\artanh(qe^{i\theta})+2\artanh(qe^{-i\theta})$.

\begin{theorem}\label{thm:parseval}
For $|q_1|,|q_2|\le1$,
\begin{equation}\label{eq:parseval}
\frac1\pi\int_0^\pi\mathcal{L}_{q_1}(\theta)\,\mathcal{L}_{q_2}(\theta)\,d\theta=8\,\chi_2(q_1q_2).
\end{equation}
The following choices identify the logarithms that occur in these moments.
\begin{enumerate}[(a)]
\item if $|q|<1$ and $\kappa=2q/(1+q^2)$, then $\mathcal{L}_q(\theta)=2\artanh(\kappa\cos\theta)$;
for $0<q<1$ this is $\log\dfrac{1+\kappa\cos\theta}{1-\kappa\cos\theta}$;
\item if $|q|=1$ and $\Re q=\lambda\in(0,1]$, then
$\Re\mathcal{L}_q(\theta)=\log\Big|\dfrac{\lambda+\cos\theta}{\lambda-\cos\theta}\Big|$;
\item if $q=i\tan\frac\eta2$ with $0<\eta<\frac\pi2$, then
$\mathcal{L}_q(\theta)=2i\arctan(\tan\eta\,\cos\theta)$.
\end{enumerate}
\end{theorem}

\begin{proof}
Orthogonality, $\int_0^\pi\cos^2j\theta\,d\theta=\pi/2$, and continuity
of the product integral on $L^2\times L^2$ give \eqref{eq:parseval}.
For $0<q<1$ the two $\artanh$ terms sum to
$\log|(1+qe^{i\theta})/(1-qe^{i\theta})|^2=2\artanh(\kappa\cos\theta)$.
This continues to $|q|<1$, since $(q+q^{-1})/2$ cannot belong to $[-1,1]$
there for nonzero $q$ and hence $\kappa$ avoids the cuts of $\artanh$.
For $q=e^{i\delta}$ the real part is the sum of the two odd cosine
series at $\theta\pm\delta$ with coefficients $2/j$. Away from multiples of
$\pi$, $\sum_{j\ {\rm odd}}\cos(jv)/j=\tfrac12\log|\cot(v/2)|$ by Abel's theorem.
The product of the two cotangents is $(\cos\delta+\cos\theta)/(\cos\delta-\cos\theta)$,
giving (b) almost everywhere and in $L^2$. In (c) substitute
$\kappa=i\tan\eta$ into (a).
\end{proof}

In (a) the factor $\tau_w$ has
\begin{equation}\label{eq:qw}
q=\frac{\sqrt w-1}{\sqrt w+1},
\end{equation}
using the principal square root for $w\in\mathbb C\setminus(-\infty,0]$.
Part (b) gives $\log|\sigma_r|$ for $\lambda=(r-1)/(r+1)$, and $\log t$ for $\lambda=1$.

\begin{corollary}\label{cor:half2}
Let $w_1,w_2>1$, and let $q_1,q_2$ be the corresponding numbers \eqref{eq:qw}. Let
$1<r\le\infty$ and
$\lambda=(r-1)/(r+1)$, where for $r=\infty$ we put $|\sigma_\infty(t)|=t$ and $\lambda=1$. Then
\begin{align}
\int_0^\infty\sqrt t\,\log\tau_{w_1}(t)\,\log\tau_{w_2}(t)\dm&=4\pi\,\chi_2(q_1q_2),\label{eq:tautau}\\
\int_0^\infty\sqrt t\,\log\tau_{w_1}(t)\,\log|\sigma_{r}(t)|\dm&=4\pi\,\Re\chi_2\big(q_1(\lambda+i\sqrt{1-\lambda^2})\big),\label{eq:tausigma}
\end{align}
and the formula \eqref{eq:min} remains true if one or both of $r_1,r_2$ are infinite. In
particular, for $w>1$ with $q$ as in \eqref{eq:qw} and for $1<r<\infty$,
\begin{gather*}
\int_0^\infty\sqrt t\,\log t\,\log\tau_w(t)\dm=4\pi\chi_2(q),\qquad
\int_0^\infty\sqrt t\,\log t\,\log|\sigma_r(t)|\dm=\pi^2\arcsin\lambda,\\
\int_0^\infty\sqrt t\,\log^2t\dm=\frac{\pi^3}{2}.
\end{gather*}
\end{corollary}

\begin{proof}
Take real parts of the factors in \eqref{eq:parseval} and use
$\overline{\mathcal L_q}=\mathcal L_{\bar q}$. This gives
\begin{equation}\label{eq:reparseval}
\frac12\int_0^\pi\Re\mathcal{L}_{q}\,\Re\mathcal{L}_{q'}\,d\theta=2\pi\,\Re\big\{\chi_2(qq')+\chi_2(q\,\overline{q'})\big\}.
\end{equation}
For regular factors take $q_j$ real. For a singular factor set
$q'=e^{i\delta}$ with $\cos\delta=\lambda$. These substitutions give
\eqref{eq:tautau} and \eqref{eq:tausigma}. For two boundary parameters use
$\Re\chi_2(e^{iv})=\tfrac\pi4(\tfrac\pi2-|v|)$, $|v|\le\pi$, from the Fourier
series of the triangular wave. The right side of \eqref{eq:reparseval} is then
$\pi^2(\pi/2-\max(\delta_1,\delta_2))$, including $\delta_j=0$.
\end{proof}

\begin{remark}\label{rem:indicator}
For $q=e^{i\delta}$ with $0\le\delta<\pi/2$, put $J=(\delta,\pi-\delta)$.
Direct integration gives the sine coefficients of $\pi\one_J$ as
$4\cos(j\delta)/j$ for odd $j$ and zero for even $j$. These are precisely
the cosine coefficients of $\Re\mathcal L_q$. Thus Parseval gives
$\int_0^\pi\Re\mathcal L_{q_1}\Re\mathcal L_{q_2}\,d\theta
=\pi^2|J_1\cap J_2|$. The intervals are concentric, which explains
the minimum law in \eqref{eq:min}.
\end{remark}

\begin{corollary}\label{cor:ladder}
For $r>1$ put $\lambda=(r-1)/(r+1)$ and $q=(\sqrt r-1)/(\sqrt r+1)$.
At $a=1/2$ the first moments and the first two diagonal products are
\[
 m_1=\pi\arcsin\lambda,\quad \mu_1=\frac{\pi^2}{2},\quad
 m_2=4\pi\chi_2(q^2),\quad \mu_2=\pi^2\arcsin\lambda .
\]
The cubic moment is $\mu_3=\pi^4/4+6\pi^2\chi_2(q^2)$.
\end{corollary}

\begin{proof}
Use \eqref{eq:first}, \eqref{eq:sfirst}, \eqref{eq:tautau}, and \eqref{eq:min},
then use $\mu_3=\pi^4/4+(3\pi/2)m_2$ from Corollary~\ref{cor:momentrelation}.
\end{proof}

At half weight, for $|\Re s|<1$, \eqref{eq:relation} gives
$H_r(\tfrac12,s)=\tan(\pi s/2)\{\pi+G_r(\tfrac12,s)\}$.
For $r=e^{2\rho}$, the first moments become
\begin{equation}\label{eq:g-first}
 \int_0^\infty\log\frac{\cosh(y+\rho)}{\cosh(y-\rho)}\frac{dy}{\sinh y}
 =\pi\gd\rho,\qquad
 \int_0^\infty\log\left|\frac{\sinh(y+\rho)}{\sinh(y-\rho)}\right|
 \frac{dy}{\sinh y}=\frac{\pi^2}{2}.
\end{equation}
For $\gamma>0$, $0<\beta<\pi/2$, Lemma~\ref{lem:trig} gives the angular
forms of $m_2$ and $\mu_3$, respectively,
\begin{align}
 \int_0^{\pi/2}\log^2\frac{\cosh\gamma+\sin u}{\cosh\gamma-\sin u}\,du
 &=4\pi\chi_2(e^{-2\gamma}),\label{eq:g-tau2}\\
 \int_0^{\pi/2} L_\beta(u)^3\frac{du}{\sin u}
 &=\frac{\pi^4}{4}+6\pi^2\chi_2(\tan^2(\beta/2)).\label{eq:g-cube}
\end{align}

\begin{corollary}\label{cor:triple}
Let $1<r_1\le r_2\le r_3$ and $q_j=(\sqrt{r_j}-1)/(\sqrt{r_j}+1)$. Then
\[
\int_0^\infty\sqrt t\,\log|\sigma_{r_1}(t)|\,\log|\sigma_{r_2}(t)|\,\log|\sigma_{r_3}(t)|\dm
=\pi^2\Big\{4\chi_2(q_1q_2)+2\chi_2(q_1q_3)+2\chi_2\Big(\frac{q_1}{q_3}\Big)\Big\}.
\]
\end{corollary}

\begin{proof}
Let $J(q,q')=\int_{1/r}^{r}\sqrt t\,\log|\sigma_{r'}(t)|\dm$,
with $q,q'$ corresponding to $r,r'$ as above. Since $I_1\subseteq I_2\subseteq I_3$,
the real part of \eqref{eq:product} at $a=1/2$, $n=3$ gives the desired
moment as $\pi^2\{J(q_2,q_1)+J(q_1,q_2)+J(q_1,q_3)\}-\pi^4/2$.
With $t=e^{2y}$, fold at $y=0$ and put $X=\tanh(y/2)$. Then
\begin{align*}
 J(q,q')&=\int_0^q\log\left|
   \frac{(X+q')(1+q'X)}{(X-q')(1-q'X)}\right|\frac{dX}{X}\\
 &=2\chi_2(qq')+
 \begin{cases}2\chi_2(q/q'),&q\le q',\\
 \pi^2/2-2\chi_2(q'/q),&q\ge q'.\end{cases}
\end{align*}
Indeed $\int_0^z\log((1+v)/(1-v))\,dv/v=2\chi_2(z)$ for $0\le z\le1$.
Above $X=q'$ invert $v=X/q'$ and use $\chi_2(1)=\pi^2/8$.
In the sum of the three $J$'s the terms involving $\chi_2(q_1/q_2)$ cancel.
\end{proof}

\begin{example}\label{ex:special}
Legendre's chi function has the classical values \cite{Campbell,Lewin}
\begin{align*}
\chi_2(1)&=\frac{\pi^2}{8}, &\chi_2(\sqrt2-1)&=\frac{\pi^2}{16}-\frac14\log^2(1+\sqrt2),\\
\chi_2(\varphi^{-1})&=\frac{\pi^2}{12}-\frac34\log^2\varphi, &\chi_2(\varphi^{-3})&=\frac{\pi^2}{24}-\frac34\log^2\varphi .
\end{align*}
The last three values follow from the classical dilogarithm evaluations at
$\pm\varphi^{-1}$ \cite{Lewin,Zagier} and
$\chi_2(x)+\chi_2((1-x)/(1+x))=\pi^2/8+
\tfrac12\log x\log((1+x)/(1-x))$ for $0<x<1$, verified by differentiation
and the limit $x\uparrow1$.
With $\cosh\gamma=\sqrt\varphi$, $e^{-2\gamma}=\varphi^{-3}$, so the
integral in \eqref{eq:g-tau2} equals $\pi^3/6-3\pi\log^2\varphi$.
Three further specializations are
\begin{align*}
\int_0^{\pi/2}\log^2\frac{\varphi^{3/2}+2\sin u}{\varphi^{3/2}-2\sin u}\,du&=\frac{\pi^3}{3}-3\pi\log^2\varphi,\\
\int_0^{\pi/2}\log^2\frac{\sqrt{1+\sqrt2}+\sqrt2\,\sin u}{\sqrt{1+\sqrt2}-\sqrt2\,\sin u}\,du&=\frac{\pi^3}{4}-\pi\log^2(1+\sqrt2),\\
\int_0^{\pi/2}\log^3\Big|\frac{2+\varphi^{3/2}\sin u}{2-\varphi^{3/2}\sin u}\Big|\,\frac{du}{\sin u}&=\frac{3\pi^4}{4}-\frac{9\pi^2}{2}\log^2\varphi ,
\end{align*}
which belong to $q^2=\varphi^{-1}$, $q^2=\sqrt2-1$ and $q^2=\varphi^{-1}$, respectively.
\end{example}

\begin{example}\label{ex:catalan}
Take $q_1=q\in(0,1]$ and $q_2=i\tan\frac\eta2$ with $0<\eta<\frac\pi2$ in
\eqref{eq:parseval}. Since $\chi_2(iv)=i\Ti_2(v)$, we get for $0<\kappa\le1$
\[
\int_0^{\pi/2}\log\frac{1+\kappa\sin u}{1-\kappa\sin u}\,\arctan(\tan\eta\,\sin u)\,du
=2\pi\,\Ti_2\Big(q\tan\frac\eta2\Big),
\]
where $q=\kappa/(1+\sqrt{1-\kappa^2})$. In the same way, for $0<\eta_1,\eta_2<\frac\pi2$,
\[
\int_0^{\pi/2}\arctan(\tan\eta_1\,\sin u)\,\arctan(\tan\eta_2\,\sin u)\,du
=\pi\,\chi_2\Big(\tan\frac{\eta_1}{2}\,\tan\frac{\eta_2}{2}\Big).
\]
At $\kappa=1$ and $\eta=\pi/6$, Ramanujan's
$\Ti_2(2-\sqrt3)=\tfrac23\beta(2)-\tfrac{\pi}{12}\log(2+\sqrt3)$ \cite{Ram1915}
gives the first integral as $\tfrac{4\pi}{3}\beta(2)-\tfrac{\pi^2}{6}\log(2+\sqrt3)$.
See also \cite{Lewin,Vlachopulos}.
\end{example}

\subsection{A transformation of Landen type}

\begin{theorem}\label{thm:landen}
Let $w>1$, $q=(\sqrt w-1)/(\sqrt w+1)$ and $w'=\dfrac{w+1}{2\sqrt w}$. For every
$s\in\mathbb{C}$,
\begin{equation}\label{eq:landen0}
G_w\big(\tfrac12,s\big)+G_w\big(\tfrac12,-s\big)=\int_0^\pi\big\{\tau_w^s-1\big\}\,d\theta=\pi\sum_{n\ge1}g_n(s)^2\,q^{2n},
\end{equation}
with $\tau_w$ as in \eqref{eq:trigG}. For $|\Re s|<1$,
\begin{equation}\label{eq:landen}
G_w\big(\tfrac12,s\big)+G_w\big(\tfrac12,-s\big)=2\sin(\pi s)\,G_{w'}(s,s).
\end{equation}
\end{theorem}

\begin{proof}
Reflection $\theta\mapsto\pi-\theta$ in \eqref{eq:trigG} exchanges $\tau_w$
and $\tau_w^{-1}$ and changes the sign of $\sec\theta$. Adding the reflected
formula gives the integral in \eqref{eq:landen0}. For
$A(z)=((1+qz)/(1-qz))^s=\sum_{n\ge0}g_n(s)q^nz^n$, using the analytic
logarithm vanishing at zero, $\tau_w^s=A(e^{i\theta})A(e^{-i\theta})$.
The integral mean over a full circle equals the mean on $[0,\pi]$
and selects the constant coefficient $\sum_{n\ge0}g_n(s)^2q^{2n}$.
Now $(w'-1)/(w'+1)=q^2$, so \eqref{eq:bilinear} gives \eqref{eq:landen}.
\end{proof}

At $s=1/2$, Theorem~\ref{thm:elliptic} evaluates the left side of
\eqref{eq:landen} and Corollary~\ref{cor:bases} evaluates its right side
since $\sqrt{1-(w')^{-2}}=\kappa$. Thus
$(1+w^{-1})K(\sqrt{1-w^{-2}})=2K(\kappa)$, Landen's transformation
\cite[(19.8.12)]{DLMF}, \cite{BB,MM}.
Write $G^\pm_w(s)=G_w(\tfrac12,s)\pm G_w(\tfrac12,-s)$.

\begin{proposition}\label{prop:legdiff}
Let $w=e^{h}>1$. For every $s\in\mathbb{C}$,
\begin{align*}
G^{+}_w(s+1)+G^{+}_w(s)&=2\pi\cosh\frac h2\,P_s(\cosh h)-2\pi,\\
G^{-}_w(s+1)-G^{-}_w(s)&=2\pi\sinh\frac h2\,P_s(\cosh h).
\end{align*}
\end{proposition}

\begin{proof}
Reflecting $\theta$ in \eqref{eq:trigG} expresses $G_w^+$ and $G_w^-$ as
integrals of $\tau_w^s-1$ against $d\theta$ and $\sec\theta\,d\theta$, respectively.
Using $1+\tau_w=2/(1-\kappa\cos\theta)$ and $(\tau_w-1)\sec\theta=\kappa(1+\tau_w)$
yields the two left sides as $2J-2\pi$ and $2\kappa J$, where
$J=\int_0^\pi\tau_w^s\,d\theta/(1-\kappa\cos\theta)$.
Set $\cos\theta=(\cos\phi+\kappa)/(1+\kappa\cos\phi)$. This changes
$\tau_w$ to $\cosh h+\sinh h\cos\phi$ and the measure to $\cosh(h/2)\,d\phi$.
Laplace's integral for $P_s$ \cite[(14.12.7)]{DLMF} now gives
$J=\pi\cosh(h/2)P_s(\cosh h)$.
\end{proof}

\section{Powers and moments of a reciprocal quadratic}\label{sec:quadratic}

Let $c>1$ and let $u,v$ be the roots of $X^2-bX+c$. By \eqref{eq:factor},
$R_{c,b}=\tau_u\tau_v$. For $b>-2\sqrt c$ the roots are positive or nonreal conjugates and
$\log R_{c,b}=\log\tau_u+\log\tau_v$ with the branches of Proposition~\ref{prop:complex}.
In particular $b=c+1,2\sqrt c,0$ gives respectively $\tau_c(t)$,
$\tau_{\sqrt c}(t)^2$ and $\tau_c(t^2)$. For $b\le-2\sqrt c$ write $u=-r_1$,
$v=-r_2$. Then $R_{c,b}=\sigma_{r_1}\sigma_{r_2}$, since $\tau_{-r}=-\sigma_r$.
The boundary gives $\sigma_{\sqrt c}^2$ and $b=-(c+1)$ gives $\sigma_c$ after
cancellation. Using $\sigma_{1/r}=1/\sigma_r$ and the constant $\sigma_1=1$
extends Theorem~\ref{thm:sigmamoments} to the first two moments of $\log|R_{c,b}|$ here.

\begin{corollary}\label{cor:quad}
Let $c>1$ and $|\Re a|<1$.
\begin{enumerate}[(i)]
\item If $b>-2\sqrt c$, put $h_1=\Log u$, $h_2=\Log v$. For $n=1,2$,
\[
 \int_0^\infty t^a\log^n R_{c,b}(t)\dm
  =\frac{\pi}{\sin\pi a}\sum_{j_1,\ldots,j_n=1}^2
   \int_0^{h_{j_1}}\!\!\cdots\!\int_0^{h_{j_n}}D_n(\xi_1,\ldots,\xi_n)\,d\xi_1\cdots d\xi_n
\]
along straight segments, with roots counted with multiplicity.
\item If $|b|<2\sqrt c$, let $q=(\sqrt u-1)/(\sqrt u+1)$ with
$u=\frac12\big(b+i\sqrt{4c-b^2}\big)$. Then
\[
\int_0^\infty\sqrt t\,\log^2R_{c,b}(t)\dm=8\pi\,\big\{\chi_2(|q|^2)+\Re\chi_2(q^2)\big\}.
\]
\item If $b=0$, the integral
$F(s)=\int_0^\infty\{R_{c,0}(t)^s-1\}\dm$ equals $\tfrac12G^-_c(s)$
for every complex $s$. With $k=\sqrt{1-c^{-2}}$,
\[
 F(1/2)=\frac{c-1}{2c}K(k),\qquad F(3/2)=(c-1)\left(E(k)+\frac{K(k)}{2c}\right).
\]
\end{enumerate}
\end{corollary}

\begin{proof}
Expand $(\log\tau_u+\log\tau_v)^n$ and apply Proposition~\ref{prop:complex} for (i).
For (ii) the angular logarithm is $\mathcal L_q+\mathcal L_{\bar q}$.
Expand its square in \eqref{eq:halfint} and apply \eqref{eq:parseval}.
For (iii) set $x=t^2$ to get $F(s)=\tfrac12\{G_c(\tfrac12,s)+G_c(-\tfrac12,s)\}$,
then use Lemma~\ref{lem:conv}(iii) and Theorem~\ref{thm:elliptic}.
\end{proof}

\noindent For two ratios $R_j=R_{c_j,b_j}$ in case (ii), with $q_j$
defined there, the same Parseval calculation gives the mixed moment
\begin{equation}\label{eq:quadratic-mixed}
 \int_0^\infty\sqrt t\,\log R_1(t)\log R_2(t)\dm
 =8\pi\Re\{\chi_2(q_1q_2)+\chi_2(q_1\overline{q_2})\}.
\end{equation}
On the diagonal, expansion of $\chi_2$ gives the positive series
$16\pi\sum_{j\ge0}(\Re(q^{2j+1}))^2/(2j+1)^2$.
Since $(\Re(q^k))^2\le|q|^{2k}$, the omitted terms $j\ge N$ sum to at most
$16\pi|q|^{4N+2}/[(2N+1)^2(1-|q|^4)]$ for every $N\ge0$.

\begin{example}\label{ex:golden}
Let $q=\varphi^{-1/2}e^{i\pi/4}$ and $u=(1+q)^2/(1-q)^2$, and let
$R_{c,b}=\tau_u\tau_{\bar u}$, so that $c=|u|^2>1$ and $b=2\,\Re u$.
Here $b/(2\sqrt c)=1-4/\sqrt5\in(-1,0)$. Then
$q$ is the number of Corollary~\ref{cor:quad}(ii), and $q^2=i\varphi^{-1}$. So
$\chi_2(q^2)=i\Ti_2(\varphi^{-1})$ is purely imaginary, and Corollary~\ref{cor:quad}(ii)
gives
\[
\int_0^\infty\sqrt t\,\log^2R_{c,b}(t)\dm=8\pi\chi_2(\varphi^{-1})=\frac{2\pi^3}{3}-6\pi\log^2\varphi .
\]
The number
$\kappa=2q/(1+q^2)$ satisfies $|\kappa|^2=4/\sqrt5$ and
$2\,\Re\kappa=2\sqrt{4+2\sqrt5}/\sqrt5$, and
$\log R_{c,b}=\mathcal{L}_q+\mathcal{L}_{\bar q}=\log\big(|1+\kappa\cos\theta|^2/|1-\kappa\cos\theta|^2\big)$.
So by \eqref{eq:halfint} the result is, in trigonometric form,
\[
\int_0^\pi\log^2\frac{4\cos^2\theta+2\sqrt{4+2\sqrt5}\,\cos\theta+\sqrt5}{4\cos^2\theta-2\sqrt{4+2\sqrt5}\,\cos\theta+\sqrt5}\,d\theta
=\frac{4\pi^3}{3}-12\pi\log^2\varphi .
\]
\end{example}

For $b>-2\sqrt c$ and $a=\pm p/q$, $0<p<q$, the evenness in $a$ of
\eqref{eq:first} and the factorization \eqref{eq:factor} give
\begin{equation}\label{eq:quad-rational-first}
 \int_0^\infty t^a\log R_{c,b}(t)\dm=J_{p/q}(u)+J_{p/q}(v).
\end{equation}
Formula~\eqref{eq:rational-logs} applies also at a repeated root.
At $a=1/3$ a single real arctangent suffices. With principal powers
$T=u^{2/3}+v^{2/3}$ and $C=c^{2/3}$, one has
\begin{equation}\label{eq:third-real}
 \int_0^\infty t^{1/3}\log R_{c,b}(t)\dm
 =2\pi\left(\frac{\pi}{3}-\arctan\frac{\sqrt3(T+1)}{T+2C-1}\right).
\end{equation}
For positive roots, add the two third-weight arctangents of
Proposition~\ref{prop:rational}, extending the single-factor formula to
$0<w<1$ by \eqref{eq:first}. Their sum lies in $(\pi/2,\pi)$ since
$T+2C-1>0$. For conjugate roots $u,v=\sqrt c\,e^{\pm i\theta}$,
$0<\theta<\pi$, one has $T=2c^{1/3}\cos(2\theta/3)>-c^{1/3}$.
Thus $T+2C-1>(c^{1/3}-1)(2c^{1/3}+1)>0$ throughout.
The symmetric sum $T$ is real analytic through the double root
(the odd powers of $\sqrt{b^2-4c}$ cancel in its Taylor expansion).
Both sides of \eqref{eq:third-real} are real analytic in $b>-2\sqrt c$,
so the same principal arctangent works on the whole interval.

\subsection{Half-integer powers at zero Mellin weight}

\begin{corollary}\label{cor:quad-half-closure}
Let $c>1$, $k=\sqrt{1-c^{-2}}$, $X=(c+c^{-1})/2$, and set
\[
 F_n(c)=\int_0^\infty \{R_{c,0}(t)^{n+1/2}-1\}\dm
 \qquad(n\in\mathbb Z).
\]
Then $F_{-n-1}=-F_n$. With $d_n=F_{n+1}-F_n$, the initial values and
recurrence are
\[
 F_0=\frac{c-1}{2c}K(k),\qquad d_{-1}=2F_0,\qquad d_0=(c-1)E(k),
\]
\[
 (2n+3)d_{n+1}=4(n+1)X d_n-(2n+1)d_{n-1}\qquad(n\geq0).
\]
As functions of $c$, the $F_n$ span exactly
$\mathbb Q(c)K(k)+\mathbb Q(c)E(k)$ over $\mathbb Q(c)$.
\end{corollary}

\begin{proof}
Corollary~\ref{cor:quad}(iii) gives $F_n=\tfrac12 G_c^-(n+\tfrac12)$,
the stated oddness, and both initial elliptic values. For the deformation $F(s)$
of that corollary, Proposition~\ref{prop:legdiff} gives
$F(s+1)-F(s)=\pi\sinh(\tfrac12\log c)P_s(X)$ for every complex $s$.
The three-term Legendre recurrence at $s=n+1/2$ \cite[\S14.10]{DLMF}
gives the displayed recurrence, so every
$F_n$ is a linear combination of $K(k)$ and $E(k)$ over
$\mathbb Q(c)$. Conversely $F_0$ and $F_1-F_0$ recover both integrals.
They are independent over $\mathbb Q(c)$ since $K(k)\sim\log(4c)$ and
$E(k)\to1$ as $c\to\infty$ \cite[\S19.12]{DLMF}.
\end{proof}

At $c=\sqrt2$ write $A=F_0(\sqrt2)$ and $B=F_1(\sqrt2)$. Then
\begin{equation}\label{eq:quadratic-pi-product}
 (32+24\sqrt2)A\{B-(1+\sqrt2)A\}=\pi .
\end{equation}
Indeed, $K_0=K(1/\sqrt2)=\Gamma(1/4)^2/(4\sqrt\pi)$ and Legendre's
relation gives $E(1/\sqrt2)=K_0/2+\pi/(4K_0)$. The initial values give
$A=(\sqrt2-1)K_0/(2\sqrt2)$ and
$B-(1+\sqrt2)A=(\sqrt2-1)\pi/(4K_0)$. Thus the entire sequence at
$c=\sqrt2$ lies in $\mathbb Q(\sqrt2)K_0+\mathbb Q(\sqrt2)\pi/K_0$.

\subsection{Even moments at half weight}

For $c>1$, $b>-2\sqrt c$ denote the two roots by $u_1,u_2$ and set
\[
 q_j=\frac{\sqrt{u_j}-1}{\sqrt{u_j}+1},\qquad
 \alpha=q_1+q_2,\qquad d=q_1q_2,\qquad
 \rho=\max(|q_1|,|q_2|).
\]
The principal roots give $0<\rho<1$ and real $\alpha,d$. Write
\[
 M_{2m}(c,b)=\int_0^\infty
      \frac{\sqrt{t}\,\log^{2m}R_{c,b}(t)}{t^2-1}\,dt,
 \qquad m\geq 1 .
\]

\begin{theorem}\label{vthree:even-quadratic}
Set \(A_0(s)=1\), \(A_j(s)=0\) for \(j<0\), and for \(n\geq0\) define
\begin{align*}
 (n+1)A_{n+1}(s)
   ={}&2\alpha s\bigl(A_n(s)-dA_{n-2}(s)\bigr)\\
     &+(\alpha^2-2d)(n-1)A_{n-1}(s)
       -d^2(n-3)A_{n-3}(s).
\end{align*}
Then \(A_n\) is a polynomial of degree at most \(n\),
\(A_n(-s)=(-1)^nA_n(s)\), and
\begin{equation}\label{vthree:even-generating}
 1+\frac{2}{\pi}\sum_{m\geq1}\frac{M_{2m}(c,b)}{(2m)!}s^{2m}
       =\sum_{n\geq0}A_n(s)^2,\qquad s\in\mathbb C .
\end{equation}
Both sides are entire in \(s\). For \(1<R<\rho^{-1}\), \(\sigma>0\)
and \(N\geq1\), put \(C_R=(1+\rho R)/(1-\rho R)\). With \([s^j]\)
denoting coefficient extraction, the truncation error satisfies
\begin{equation}\label{vthree:even-error}
 \left|M_{2m}(c,b)-\frac{\pi(2m)!}{2}[s^{2m}]
             \sum_{n=0}^{N-1}A_n(s)^2\right|
 \leq \frac{\pi(2m)!}{2\sigma^{2m}}\,
         C_R^{4\sigma}\frac{R^{-2N}}{1-R^{-2}} .
\end{equation}
\end{theorem}

\begin{proof}
Let \(L\) be the analytic logarithm on \(|z|<\rho^{-1}\), with
\(L(0)=0\), of
\[
 C(z)=\prod_{j=1}^2\frac{1+q_jz}{1-q_jz}
     =\frac{1+\alpha z+dz^2}{1-\alpha z+dz^2}.
\]
Set \(F_s(z)=\exp(sL(z))\).
Logarithmic differentiation gives
\[
 \bigl(1+(2d-\alpha^2)z^2+d^2z^4\bigr)F_s'(z)
       =2\alpha s(1-dz^2)F_s(z).
\]
Comparison of coefficients gives the stated recurrence for
\(A_n(s)=[z^n]F_s(z)\) and, by induction, its degree bound.
Parity follows from \(F_{-s}(z)=F_s(-z)\). For algebraic $c,b$ the
recurrence gives algebraic terms for each series $M_{2m}/\pi$.

Set $\ell(\theta)=\log R_{c,b}(t)$ at $(t-1)/(t+1)=\cos\theta$.
By the factorization in Theorem~\ref{thm:parseval}(a),
$\ell=L(e^{i\theta})+L(e^{-i\theta})$, with branches fixed by $L(0)=0$
and the real coefficients of $C$. Reciprocity gives $\ell(\pi-\theta)=-\ell(\theta)$,
so Lemma~\ref{lem:trig} gives $2M_{2m}=\int_0^\pi\ell^{2m}\,d\theta$.
Hence the left side of \eqref{vthree:even-generating} is the full-circle mean of
$F_s(e^{i\theta})F_s(e^{-i\theta})$. Its constant coefficient is
$\sum_n A_n(s)^2$, as in Theorem~\ref{thm:landen}.

On \(|z|=R\), the odd-power series for each
\(\log((1+q_jz)/(1-q_jz))\) gives \(|L(z)|\leq2\log C_R\).
Cauchy's estimate yields \(|A_n(s)|\leq C_R^{2\sigma}R^{-n}\) on
\(|s|\leq\sigma\), justifying termwise integration and local uniform
convergence. Summing the squared bounds for \(n\geq N\) and applying
Cauchy's estimate on \(|s|=\sigma\) gives \eqref{vthree:even-error}.
\end{proof}

\subsection{Odd moments on the singular boundary}

At \(b=-2\sqrt c\) write \(r=\sqrt c>1\), so that
\(R_{c,b}=\sigma_r^2\). The following formula gives the individual
moments of \(\log|\sigma_r|\) at zero Mellin weight.

\begin{theorem}\label{vthree:odd-boundary}
Let \(x=1/r\) and \(h=\log r\), and write
\(\chi_j(z)=\sum_{k\geq0}z^{2k+1}/(2k+1)^j\) wherever the series
converges. For every odd integer \(n\geq1\),
\[
 \int_0^\infty\frac{\log^n|\sigma_r(t)|}{t^2-1}\,dt
 =2n!\left(\chi_{n+1}(1)+
        \sum_{k=0}^{n}\frac{h^k}{k!}\chi_{n+1-k}(x)\right).
\]
The even moments of positive order vanish. For the corresponding
moment of \(\log R_{c,b}\), multiply the right side by \(2^n\).
The generating function, for $|\Re s|<1$, is
\begin{equation}\label{eq:boundary-lerch}
 2H_r(0,s)=\pi\tan(\pi s/2)
 +x^{1-s}\Phi(x^2,1,(1-s)/2)-x^{1+s}\Phi(x^2,1,(1+s)/2),
\end{equation}
where $\Phi(z,1,v)=\sum_{j\ge0}z^j/(j+v)$ \cite[\S25.14]{DLMF}.
\end{theorem}

\begin{proof}
The zero and pole of \(\sigma_r\) give integrable logarithmic
singularities. Inversion \(t\mapsto1/t\) gives the assertion for even
orders. For odd \(n\), substitute \(u=\sigma_r(t)\), splitting at
\(t=r\). The two ranges, in order of increasing \(u\), are
\((-\infty,x)\) and \((r,\infty)\), with the change of sign in the
differential cancelling the reversal of orientation. Reflect the
negative half-line and invert the tails of the resulting integrals.
The moment becomes
\[
 2\left(\int_0^1+\int_0^x\right)
       \frac{(-\log u)^n}{1-u^2}\,du .
\]
Expand \((1-u^2)^{-1}\) by monotone convergence. For any positive
integer \(j\) and \(0<y\leq1\), integration by parts gives
\[
 \int_0^y u^{j-1}(-\log u)^n\,du
   =n!y^j\sum_{k=0}^{n}\frac{(-\log y)^k}{k!\,j^{n+1-k}} .
\]
Summing over odd \(j\) at \(y=1,x\) proves the moment formula.
The same substitution before differentiation gives
\[
 H_r(0,s)=\int_0^\infty\frac{u^s-1}{u^2-1}\,du
          +\int_0^x\frac{u^{-s}-u^s}{1-u^2}\,du .
\]
Lemma~\ref{lem:tan} evaluates the first integral. In the second, the
geometric series is uniformly convergent for $0\le u\le x<1$, and the
powers are integrable on the stated strip. Termwise integration gives
\eqref{eq:boundary-lerch}.
\end{proof}

\section{The exact field of rational-weight sampling}\label{sec:sampling}
\label{geo:section}
Fix $c>1$ and $b>-2\sqrt c$, and let $R_{c,b}$ be the quadratic ratio in
\eqref{eq:factor}. Define its deformation and sample series by
\begin{equation}\label{eq:deformation-def}
 G_{c,b}(a,s)=\int_0^\infty t^a\{R_{c,b}(t)^s-1\}\dm,\qquad
 \mathcal S_a(z)=\sum_{k\ge1}G_{c,b}(a,k)z^k,\qquad |\Re a|<1.
\end{equation}
Use the real logarithm of the positive, bounded $R_{c,b}$. Since
$R_{c,b}-1=(c-1)(t^2-1)/(t^2+bt+c)$, the singularity at $1$ is removable.
When $b=c+1$ this is the transform of $\tau_c$. Curves below are smooth
projective models over $\mathbb C$. Write
\begin{equation}\label{eq:sampling-ABCD}
	A=1-cz,\quad B=b(1-z),\quad C=c-z,\quad D=B^2-4AC.
\end{equation}
For $-1<\Re a<1$, define
\begin{equation}\label{eq:sampling-SI}
	I_a(z)=\int_0^\infty\frac{t^a\,dt}{At^2+Bt+C}.
\end{equation}
For sufficiently small $|z|$, uniform geometric summation in
\eqref{eq:deformation-def} gives
\begin{equation}\label{eq:sampling-master}
	\mathcal S_a(z)=\frac{(c-1)z}{1-z}I_a(z).
\end{equation}
Indeed, the factor $R_{c,b}(t)-1$ cancels $t^2-1$ before integration,
and a constant multiple of $t^{\Re a}/(1+t^2)$ dominates the sum.
The classical beta integral and partial fractions give
\begin{equation}\label{eq:sampling-root-form}
	I_a=-\frac{\pi}{\sin\pi a\sqrt D}
	\{(-r_+)^a-(-r_-)^a\},\qquad
	r_\pm=\frac{-B\pm\sqrt D}{2A}.
\end{equation}
At $z=0$ use the principal logarithms of $-r_\pm$, then continue the
germs. Partial fractions converge separately for $-1<\Re a<0$;
holomorphy gives the full strip. At $a=0$ use the logarithmic limit.
At a finite nonzero double root the confluent divided difference applies
when the powers use the same local logarithm, in particular for the
principal germ if $D(0)=0$. For the Cayley kernel take $(c,b)=(5,-2)$ and real
$0<a<1$. Here $D=-16(z^2-6z+1)$. As $z\uparrow z_* =3-2\sqrt2$ the
conjugate roots meet at $r_*=-B(z_*)/(2A(z_*))>0$. Their negatives reach
opposite sides of the logarithmic cut, so
$I_a(z)\sim 2\pi r_*^a/\sqrt{-D(z)}$ by \eqref{eq:sampling-root-form}.

\begin{proposition}[Algebraicity and minimal differential order]
	\label{thm:sampling-algebraicity-iff}
	Assume $c>1$, $b>-2\sqrt c$, and $-1<\Re a<1$.
	For $b\ne0$, $\mathcal S_a$ is algebraic over $\mathbb C(z)$ exactly when
	$a\in\mathbb Q\setminus\{0\}$, and its minimal homogeneous differential
	order over $\mathbb C(z)$ is two. For $b=0$, algebraicity holds exactly when
	$a\in\mathbb Q$, and the minimal order is one.
\end{proposition}
\begin{proof}
For $ab\ne0$ the logarithmic derivatives of $f_\pm=(-r_\pm)^a$
are $a(\lambda\pm\delta)/2$, where Vieta's formulas give
\[
 \lambda=\frac{c^2-1}{AC},\qquad
 \delta=-\frac{b(c-1)^2(1+z)}{AC\sqrt D},\qquad \eta=\delta'/\delta .
\]
Both powers satisfy
\[
 f''-(a\lambda+\eta)f'
 +\left[-\tfrac a2\lambda'+\tfrac{a^2}4(\lambda^2-\delta^2)
                  +\tfrac a2\lambda\eta\right]f=0 .
\]
The coefficients are rational in $z$. Multiplication by
$z/((1-z)\sqrt D)$ preserves the order since its logarithmic
derivative is rational. The resulting equation for $\mathcal S_a$
extends to $a=0$ by holomorphy, its coefficients being polynomial in $a$.

At $z=1/c$ one root has a simple pole and the other is finite and nonzero.
Up to a holomorphic unit, $\mathcal S_a=(1-cz)^{-a}u-v$ for local units
$u,v$. For $a\ne0$ the two terms are independent since $a\notin\mathbb Z$.
A circuit multiplies the first by $e^{-2\pi ia}$ and fixes the second,
excluding an order-one equation and giving infinitely many branches
unless $a\in\mathbb Q$. For nonzero rational $a$, \eqref{eq:sampling-root-form}
is algebraic. At $a=0$ the logarithmic limit has infinite additive monodromy
and likewise excludes order one. For $b=0$, \eqref{eq:sampling-bzero-I}
is nonzero with rational logarithmic derivative and exponent
$-(a+1)/2$ at $1/c$, giving both assertions.
\end{proof}

The ambient curve $u^q+v^q=\beta(1+u^qv^q)$, $\beta=b/(c+1)$, is known
\cite[Eq.~(5)]{Roan}. The sample generates the following quotient.

\begin{theorem}[Degree and genus]\label{geo:degreegenus}
	Let $c>1$, $b>-2\sqrt c$, and
	\begin{equation}\label{geo:nondegenerate}
		b\ne0,\qquad b^2\ne(c+1)^2.
	\end{equation}
	Let $a=p/q$, where $p\in\mathbb Z\setminus\{0\}$, $q\geq2$,
	$\gcd(p,q)=1$, and $|p|<q$. Then
	\begin{equation}\label{geo:degreegenusformula}
		[\mathbb C(z,I_{p/q}):\mathbb C(z)]=q^2,
		\qquad
		g\bigl(\mathbb C(z,I_{p/q})\bigr)=\frac{(q-1)(q-2)}2.
	\end{equation}
	With the Dickson polynomials
\begin{equation}\label{geo:Dickson}
		\mathcal D_0(S,P)=2,\quad \mathcal D_1(S,P)=S,\quad
		\mathcal D_n=S\mathcal D_{n-1}-P\mathcal D_{n-2}.
	\end{equation}
	the smooth projective plane model is the homogenization of
\begin{equation}\label{geo:planecurve}
		(c+1)\mathcal D_q(S,P)=b(P^q+1).
	\end{equation}
\end{theorem}

\begin{proof}
For a root $t$ of $At^2+Bt+C=0$, the degree-two map
$z=(t^2+bt+c)/(ct^2+bt+1)$ has deck involution
\begin{equation}\label{geo:involution}
  \tau(t)=-\frac{(c+1)t+b}{bt+c+1},\qquad
  (c+1)(t+\tau(t))+b(t\tau(t)+1)=0 .
\end{equation}
The zeros and poles of $t,\tau(t)$ are the four distinct points
$0,\infty,-b/(c+1),-(c+1)/b$ under \eqref{geo:nondegenerate}. Set
\begin{equation}\label{geo:Kummer}
 u^q=-t,\quad v^q=-\tau(t),\quad K=\mathbb C(t,u,v).
\end{equation}
Valuations show $t^i\tau(t)^j\in\mathbb C(t)^{\times q}$ only if
$q\mid i,j$. Thus the Kummer group is $\mu_q^2$ with inertia order $q$
at each of the four points and no other ramification. Riemann--Hurwitz gives
\begin{equation}\label{geo:Kgenus}
 2g(K)-2=-2q^2+4q^2(1-1/q),\qquad g(K)=(q-1)^2 .
\end{equation}
The swap $\sigma(u,v)=(v,u)$ lifts $\tau$, so the group over $\mathbb C(z)$
is $\mu_q^2\rtimes\langle\sigma\rangle$. Choose $d=A(t-\tau(t))$, so
$d^2=D$ and $\sigma(d)=-d$. On compatible branches \eqref{eq:sampling-root-form} gives
\begin{equation}\label{geo:samplefield}
 I_{p/q}=-\frac{\pi}{\sin(\pi p/q)}\,\frac{u^p-v^p}{d}.
\end{equation}
The swap fixes this function. A rotation fixing it satisfies
$(\zeta_1^p-1)u^p=(\zeta_2^p-1)v^p$. The ratio $u^p/v^p$ is nonconstant,
so coprimality forces both rotations to be trivial, also for $p<0$. Hence
\begin{equation}\label{geo:fixedfield}
 \mathbb C(z,I_{p/q})=K^{\langle\sigma\rangle},\qquad
 [\mathbb C(z,I_{p/q}):\mathbb C(z)]=q^2 .
\end{equation}
The two roots of $bt^2+2(c+1)t+b=0$ are the fixed points of $\tau$.
Both are unramified in $K$ and above each there are $q$ fixed points
$u=v$ of $\sigma$. Thus $2(q-1)^2-2=2(2g-2)+2q$, giving the stated genus.

With $S=u+v$, $P=uv$, the equation $X^2-SX+P=0$ gives
$K^{\langle\sigma\rangle}=\mathbb C(S,P)$. Since
$\mathcal D_q(S,P)=u^q+v^q$, \eqref{geo:involution} gives \eqref{geo:planecurve}.
The latter polynomial is irreducible: substitution of $S,P$ would take
a factorization to one of $bu^qv^q-(c+1)(u^q+v^q)+b$. This relation
is irreducible over $\mathbb C(u)$ by Eisenstein at any simple zero
of $-\tau(-u^q)$, and its coefficients in $\mathbb C[u]$ are coprime.
The degree-$q$ plane model has arithmetic genus
$(q-1)(q-2)/2$, equal to the geometric genus, hence is smooth.
\end{proof}

\begin{remark}[Exceptional parameters]\label{geo:exceptions}
	For $b=0$, Euler's beta integral gives
	\begin{equation}\label{eq:sampling-bzero-I}
		I_a=\frac{\pi}{2\cos(\pi a/2)}C^{(a-1)/2}A^{-(a+1)/2}.
	\end{equation}
	For $a=p/q$ as above its degree is $q$ if $p,q$ are both odd and
	$2q$ otherwise, and its genus is zero: after a rational factor it is
	$(C/A)^{(p-q)/(2q)}$, with just two ramification points.
	When $b=c+1$, cancellation of $t+1$ gives
	\begin{equation}\label{eq:sampling-cancelled}
		\mathcal S_a=\frac{\pi z}{(1-z^2)\sin\pi a}
		\{(C/A)^a-1\},
	\end{equation}
	of degree $q$ and genus zero. The case $b=-(c+1)$ violates the hypotheses.
At $b=2\sqrt c$ the principal germ at $z=0$ uses the confluent limit in
\eqref{eq:sampling-root-form}.
\end{remark}

\section{Cubic models and normalized CM periods}
\label{geo:ellipticsection}

Retain \eqref{geo:nondegenerate} and set
\begin{equation}\label{geo:Vbeta}
	V(z)=\frac{\sqrt3}{2\pi}I_{1/3}(z),\qquad \beta=\frac b{c+1}.
\end{equation}

\begin{theorem}[Nonic, cubic, and differential]\label{geo:noniccubic}
	The distinguished germ $V$ satisfies the degree-nine equation
	\begin{equation}\label{geo:nonic}
		(1-AD V^3)^3=27A^2C V^3.
	\end{equation}
	Its smooth projective curve is birational to
	\begin{equation}\label{geo:cubic}
		\mathcal E_\beta:\quad S^3-3SP-\beta(P^3+1)=0.
	\end{equation}
	On the curve the identities
	\begin{equation}\label{geo:VP}
		P^3=\frac CA,\qquad V=\frac1{A(S^2-P)}
	\end{equation}
	hold. A nonzero holomorphic differential is
	\begin{equation}\label{geo:omega}
		\omega=\frac{dP}{S^2-P}=AV\,dP
		=\frac{(c^2-1)V}{3AP^2}\,dz.
	\end{equation}
	These are meromorphic identities on the smooth projective curve.
\end{theorem}

\begin{proof}
Choose $u^3=-r_+$, $v^3=-r_-$ as in \eqref{geo:samplefield}. With
$S=u+v$, $P=uv$, we have $P^3=C/A$ and
$(u-v)(S^2-P)=-\sqrt D/A$. Since $V=-(u-v)/\sqrt D$, this proves
\eqref{geo:VP}. Since $ADV^3=(u-v)^2/(S^2-P)$, one has
$1-ADV^3=3APV$. Cubing gives \eqref{geo:nonic};
Theorem~\ref{geo:degreegenus} gives irreducibility and
\eqref{geo:planecurve} gives the cubic. The inverse is
\begin{equation}\label{geo:inverseVP}
 P=\frac{1-ADV^3}{3AV},\qquad
 S=\frac{\beta(P^3+1)}{1/(AV)-2P},\qquad z=\frac{c-P^3}{1-cP^3}.
\end{equation}
The plane cubic residue differential $dP/[3(S^2-P)]$ is holomorphic
and nonzero. Differentiating $P^3=C/A$ now gives \eqref{geo:omega}.
\end{proof}

\begin{theorem}[A normalized Weierstrass model]\label{geo:Weierstrass}
	Put
	\begin{equation}\label{geo:linearcoords}
		Z=S-\beta(P+1),\qquad X=\frac{P+1}{Z},\qquad
		Y=\frac{P-1}{Z},\qquad
		\eta=\frac{4\beta(1-\beta^2)}3,\qquad \xi=1-4\beta^2.
	\end{equation}
	Then $x=\eta X+\xi/3$, $y=\eta Y$ transform
	\eqref{geo:cubic} birationally to
	\begin{equation}\label{geo:Weierstrasseq}
		y^2=x^3-\frac{1+8\beta^2}{3}x
		+\frac{2(1-20\beta^2-8\beta^4)}{27},
	\end{equation}
	and preserve the differential in the precise normalization
	\begin{equation}\label{geo:differentialscale}
		\frac{dx}{y}=\frac{dX}{Y}=\frac{dP}{S^2-P}=\omega.
	\end{equation}
	The discriminant and $j$-invariant are
	\begin{equation}\label{geo:j}
		\Delta_{\!E}=\frac{4096}{27}\beta^2(1-\beta^2)^3,
		\qquad
		j(\mathcal E_\beta)=
		\frac{27(1+8\beta^2)^3}{\beta^2(1-\beta^2)^3}.
	\end{equation}
\end{theorem}

\begin{proof}
Substitution of \eqref{geo:linearcoords} gives
$Y^2=\eta X^3+\xi X^2-4\beta X-4/3$. Translation and scaling give
\eqref{geo:Weierstrasseq} with inverse
\begin{equation}\label{geo:Weierstrassinverse}
 X=(x-\xi/3)/\eta,\quad Y=y/\eta,\quad
 P=(X+Y)/(X-Y),\quad S=2(1+\beta X)/(X-Y).
\end{equation}
Differentiating the cubic gives $(S^2-P)dS=(S+\beta P^2)dP$. Substitution
in $dX/Y=[S\,dP-(P+1)dS]/[Z(P-1)]$ gives \eqref{geo:differentialscale}.
The short Weierstrass discriminant and $j$ formulas give \eqref{geo:j}.
\end{proof}

Writing $x_0=S/\beta^{1/3}$, $y_0=-P$, $w_0=-W$ gives the Hesse pencil
$x_0^3+y_0^3+w_0^3=3\kappa x_0y_0w_0$, $\kappa^3=\beta^{-2}$
\cite{ArtebaniDolgachev,McMullenMukamelWright}; see also
\cite{DaviesNeeman} for the three-state chiral Potts curve.
For $(c,b)=(5,-2)$, \eqref{geo:j} gives $j=1193859/512$.
CM periods have classical gamma evaluations \cite{ChowlaSelberg}.

\begin{theorem}[An exact CM period of the sampling curve]\label{geo:CMperiod}
	Let
	\begin{equation}\label{geo:CMparameters}
		\beta^2=\frac{3\sqrt3-5}{4}=\frac1{10+6\sqrt3},\qquad
		\alpha=2\sqrt3-3>0.
	\end{equation}
	Assume $(c,b)$ is an admissible pair with $b/(c+1)=\beta$.
	The model \eqref{geo:Weierstrasseq} becomes
	\begin{equation}\label{geo:lemniscatic}
		y^2=x^3-\alpha x,
	\end{equation}
	with $j=1728$. Pull back to the cubic the cycle from $x=\sqrt\alpha$
to infinity on $y>0$ and back on $y<0$, using \eqref{geo:Weierstrassinverse}.
For the resulting cycle $\Gamma$,
\begin{equation}\label{geo:GammaPeriod}
		\boxed{\displaystyle
			\oint_\Gamma AV\,dP
			=\oint_\Gamma\frac{dP}{S^2-P}
			=\frac{\Gamma(1/4)^2}
			{\sqrt{2\pi}\,(2\sqrt3-3)^{1/4}}.}
	\end{equation}
	One realization is $b=-2$, $c=2\sqrt{10+6\sqrt3}-1$.
\end{theorem}

\begin{proof}
At \eqref{geo:CMparameters} the constant term of \eqref{geo:Weierstrasseq}
vanishes and its linear coefficient is $-\alpha$. The automorphism
$(x,y)\mapsto(-x,iy)$ gives CM by $i$. By \eqref{geo:differentialscale}
and the chosen orientation, scaling $x=\sqrt\alpha\,u$ and setting $v=u^{-2}$ give
\[
 \oint_\Gamma\omega=2\alpha^{-1/4}\int_1^\infty\frac{du}{\sqrt{u^3-u}}
 =\alpha^{-1/4}B(1/4,1/2)
 =\frac{\Gamma(1/4)^2}{\sqrt{2\pi}\,\alpha^{1/4}},
\]
using gamma reflection \cite{DLMF}. The stated pair gives $\beta^2$ as required.
\end{proof}

For algebraic $c,b$ and rational $a\ne0$, \eqref{eq:sampling-root-form}
makes $\mathcal S_a/\pi$ algebraic over $\overline{\mathbb Q}(z)$.
Its polynomially weighted Taylor sums are algebraic at algebraic points
inside the disk for weights in $\overline{\mathbb Q}[n]$.
The following period series vary the cubic.

\section{The full period law and its hypergeometric consequences}
\label{xp:section}

In the Picard--Fuchs description \cite{Maier2009}, use the differential of
Theorem~\ref{geo:Weierstrass} and write
\begin{equation}\label{xp:family}
	\rho=\beta^2,\qquad
	f(x,\rho)=x^3-\frac{1+8\rho}{3}x+\frac{2(1-20\rho-8\rho^2)}{27},\qquad
	\omega=\frac{dx}{\sqrt{f(x,\rho)}}=AV\,dP.
\end{equation}
Set $F(\rho)={}_2F_1(1/3,2/3;1;\rho)$,
$\mathcal K(m)=K(\sqrt m)$ and $\mathcal E(m)=E(\sqrt m)$.
Use positive square roots in the real integrals below.

\begin{theorem}[Periods with an exact differential certificate]
	\label{xp:periodlaw}
	For $0<\rho<1$, let $e_1>e_2>e_3$ be the roots of $f(x,\rho)$, and put
	$R=e_1-e_3$, $m=(e_1-e_2)/R$.  Orient the real cycle above
	$[e_1,\infty]$ positively and the cycle above $[e_2,e_1]$ so that its
	period has positive imaginary part.  Their periods are
	\begin{align}
		\Omega_r(\rho)&=\frac{4\mathcal K(1-m)}{\sqrt R}
		=\frac{2\pi}{\sqrt3}F(1-\rho),\label{xp:realperiod}\\
		\Omega_i(\rho)&=\frac{4\mathrm{i} \mathcal K(m)}{\sqrt R}
		=2\pi\mathrm{i} F(\rho).\label{xp:imagperiod}
	\end{align}
	Every locally transported period satisfies
	\begin{equation}\label{xp:PF}
		\rho(1-\rho)\Omega''+(1-2\rho)\Omega'-\frac29\Omega=0.
	\end{equation}
	An exact certificate for this equation is
	\begin{equation}\label{xp:certificate}
		\left[\rho(1-\rho)\partial_\rho^2+(1-2\rho)\partial_\rho-\frac29\right]
		\frac{dx}{\sqrt f}
		=d_x\left(\frac{T(x,\rho)}{f(x,\rho)^{3/2}}\right),
	\end{equation}
	where
	\begin{equation}\label{xp:T}
		T=\frac49x^4-\frac{4(4\rho+5)}{27}x^2
		+\frac{8(4\rho^2-26\rho-5)}{243}x+\frac{32(\rho-1)^2}{243}.
	\end{equation}
\end{theorem}

\begin{proof}
The discriminant $256\rho(1-\rho)^3/27$ gives three distinct real roots.
Clearing $f^{5/2}$ in \eqref{xp:certificate} reduces it to the polynomial identity
\[
 \rho(1-\rho)(3f_\rho^2/4-ff_{\rho\rho}/2)
 -(1-2\rho)ff_\rho/2-2f^2/9=T_x f-3Tf_x/2 ,
\]
verified by substitution of \eqref{xp:T}. Represent a closed cycle away
from the branch points and infinity. Locally in $\rho$ its projected paths
can be held fixed and lifted continuously as the simple roots move.
Differentiation under these path integrals is valid, and the integral of the
single-valued exact differential over the closed cycle vanishes, giving
\eqref{xp:PF}. The usual cubic substitutions give the elliptic-integral
forms of \eqref{xp:realperiod}--\eqref{xp:imagperiod}. The cycles over
the adjacent branch intervals intersect once after a transverse displacement
and form an integral homology basis.

As $\rho\downarrow0$, the roots tend to $1/3,1/3,-2/3$, whence
$\Omega_i/\mathrm i\to2\pi$. At zero the hypergeometric equation
\eqref{xp:PF} has normalized holomorphic solution $F$ and a logarithmic
second solution. This fixes $\Omega_i$. As $\rho\uparrow1$, the roots
tend to $2,-1,-1$, giving $\Omega_r\to2\pi/\sqrt3$; invariance under
$\rho\mapsto1-\rho$ fixes $\Omega_r$.
\end{proof}

\begin{corollary}[A second CM period in the real kernel family]
	\label{xp:CMthird}
	At $\rho=1/2$, the curve has $j=54000$, period ratio $\mathrm{i}\sqrt3$, and
	complex multiplication by the order $\mathbb Z[\sqrt{-3}]$ of
	discriminant $-12$.  Its periods are
	\begin{equation}\label{xp:CMthirdvalues}
		\Omega_r=\frac{\sqrt3\,\Gamma(1/3)^3}{2^{4/3}\pi},\qquad
		\Omega_i=\frac{3\mathrm{i}\,\Gamma(1/3)^3}{2^{4/3}\pi}.
	\end{equation}
	This parameter occurs at $b=-2$, $c=2\sqrt2-1$, which satisfies
	the kernel conditions $c>1$ and $b>-2\sqrt c$.
\end{corollary}

\begin{proof}
Gauss's second summation \cite[Eq.~15.4.28]{DLMF} and gamma duplication
and reflection give $F(1/2)=\sqrt\pi/[\Gamma(2/3)\Gamma(5/6)]
=3\Gamma(1/3)^3/(2^{7/3}\pi^2)$. Theorem~\ref{xp:periodlaw} gives both
periods. The lattice $\mathbb Z+\mathrm i\sqrt3\,\mathbb Z$ is a ring;
any complex scalar preserving it belongs to it by evaluation at $1$.
Equation \eqref{geo:j} gives $j=54000$, and the stated pair gives
$\beta=-1/\sqrt2$.
\end{proof}

\subsection{Three companion series at the lemniscatic parameter}
\label{xp:companionsubsection}
Set
\begin{equation}\label{xp:constants}
	s=\sqrt3,\quad \rho_0=\frac{3s-5}{4},\quad
	\alpha=2s-3,\quad z_0=4\rho_0(1-\rho_0)=\frac32(7s-12),
\end{equation}
and define
\begin{equation}\label{xp:H}
	a_n=\frac{(1/2)_n(1/3)_n(2/3)_n}{(n!)^3},\qquad
	H(z)=\sum_{n=0}^{\infty}a_nz^n,\qquad \vartheta=z\frac d{dz}.
\end{equation}
Here $0<\rho_0<1/2$, $0<z_0<1$. For the modular method and quadratic
irrational arguments see \cite{Ram1914,BB,BBG,CCY,GuilleraZudilin}.
The base value is the square of the Berndt--Chan evaluation
\cite[p.~280]{BerndtChan}, in the gamma form recorded by Shpot
\cite[Eq.~(27)]{Shpot}. The linear evaluation is recorded by Aldawoud
\cite[p.~5]{Aldawoud} with $x=z_0/108$,
$\lambda=1/(5+\sqrt3)$ and $a_n=(3n)!(2n)!/[108^n(n!)^5]$.
The quadratic companion follows from the symmetric square. For logarithmic
companions see \cite{GuilleraMatrix} and
\cite[Lemma~3.2, Eq.~(3.16)]{ZhouHarmonic}.

\begin{theorem}[Three exact companion evaluations]
	\label{xp:threecompanions}
	With the notation above,
	\begin{align}
		\sum_{n=0}^{\infty}a_nz_0^n
		&=\frac{\Gamma(1/4)^4}{8\pi^3\sqrt\alpha},\label{xp:gammafourth}\\
		\sum_{n=0}^{\infty}a_n\{1+(5+s)n\}z_0^n
		&=\frac{2+s}{\pi},\label{xp:knownseries}\\
		\sum_{n=0}^{\infty}a_n
		\{1452n^2+(600-318s)n+102-71s\}z_0^n
		&=\frac{(912+504s)\sqrt\alpha\,\pi}{\Gamma(1/4)^4}.
		\label{xp:reciprocalgamma}
	\end{align}
	All three series converge absolutely.
\end{theorem}

\begin{proof}
The quadratic transformation and Clausen's identity
\cite[Eqs.~15.8.18 and 16.12.2]{DLMF} give, on the branches normalized at zero,
\begin{equation}\label{xp:Clausen}
 F(\rho)={}_2F_1(1/6,1/3;1;4\rho(1-\rho)),\qquad
 F(\rho)^2=H(4\rho(1-\rho)),\quad 0<\rho<\tfrac12 .
\end{equation}
At $\rho_0$, \eqref{geo:lemniscatic} gives $R=2\sqrt\alpha$, $m=1/2$.
Theorem~\ref{xp:periodlaw} and $\mathcal K(1/2)=\Gamma(1/4)^2/(4\sqrt\pi)$
give \eqref{xp:gammafourth}. Implicitly differentiating $f(e_j(\rho),\rho)=0$ yields
\begin{equation}\label{xp:rootderivatives}
 R'(\rho_0)=8/(3\sqrt\alpha),\qquad m'(\rho_0)=2\sqrt3/(3\alpha^{3/2}).
\end{equation}
Legendre's relation \cite[Eq.~19.7.1]{DLMF} and
$\partial_m\mathcal K=\mathcal E/[2m(1-m)]-\mathcal K/(2m)$ give
$(\mathcal K\,\partial_m\mathcal K)(1/2)=\pi/2$. Differentiating
$F=2\mathcal K(m)/(\pi\sqrt R)$ at $\rho_0$ therefore gives
$FF'+R'F^2/(2R)=2m'/(\pi R)$, or by \eqref{xp:rootderivatives},
\begin{equation}\label{xp:CMderivative}
 (2-\sqrt3)F(\rho_0)^2+z_0F(\rho_0)F'(\rho_0)=1/\pi .
\end{equation}
As $FF'=2(1-2\rho)H'(4\rho(1-\rho))$, this yields \eqref{xp:knownseries}.

For the quadratic companion write $\mathcal W=\rho(1-\rho)$,
$h=\mathcal W/(1-2\rho)$. With $\vartheta=h\partial_\rho$ on $H(4\mathcal W)$,
\eqref{xp:PF} and \eqref{xp:Clausen} give
\begin{equation}\label{xp:nonlinear}
 \vartheta^2H=(h'-1)\vartheta H+\frac{(\vartheta H)^2}{2H}
                   +\frac{4h^2}{9\mathcal W}H .
\end{equation}
Put $B_0=5+s$, $C_0=2+s$. Substitute
$\vartheta H=(C_0/\pi-H)/B_0$ at $z_0$, where $H>0$, and collect terms:
\[
 (\vartheta^2+u_0\vartheta+v_0)H(z_0)
 =\frac{C_0^2}{2B_0^2\pi^2H(z_0)} .
\]
Here $u_0=1-h'(\rho_0)+1/B_0$ and
$v_0=1/(2B_0^2)-4h(\rho_0)^2/[9\mathcal W(\rho_0)]$.
Inserting \eqref{xp:constants} gives $1452u_0=600-318s$,
$1452v_0=102-71s$, $5808C_0^2/B_0^2=912+504s$.
Now \eqref{xp:gammafourth} gives \eqref{xp:reciprocalgamma}.
The ratio $a_{n+1}/a_n\to1$ and $z_0<1$ ensure absolute convergence.
\end{proof}

The coefficient ratio gives
$[\vartheta^3-z(\vartheta+1/2)(\vartheta+1/3)(\vartheta+2/3)]H=0$.
Since $z_0\ne1$, every higher derivative at $z_0$ reduces to
$H,\vartheta H,\vartheta^2 H$ over $\mathbb Q(\sqrt3)$. Thus the sums
with weights in $\mathbb Q(\sqrt3)[n]$ belong to
\[
 \operatorname{span}_{\mathbb Q(\sqrt3)}\{T,1/\pi,1/(\pi^2T)\},
 \qquad T=H(z_0).
\]

\section{The quarter-weight curve: a CM decomposition and sampling periods}
\label{xq:section}

The Ciani splitting \cite{Guardia,KaniRosen} gives three elliptic
quotients. We determine the differential pullbacks and cycle multiplicities.

\begin{theorem}[Three explicit elliptic quotients]
	\label{xq:quotients}
	Let $\beta\in\mathbb C\setminus\{0,1,-1\}$.  The projective quartic
	\begin{equation}\label{xq:curve}
		\mathcal C_\beta:\quad
		S^4-4S^2PW+2P^2W^2-\beta(P^4+W^4)=0
	\end{equation}
	is smooth and has genus three.  On the chart $W=1$, put
	\begin{equation}\label{xq:differentials}
		Y=S^2-2P,\qquad
		\eta_0=\frac{dP}{SY},\quad
		\eta_1=\frac{P\,dP}{SY},\quad
		\eta_2=\frac{dP}{Y}.
	\end{equation}
	These differentials form a basis of $H^0(\mathcal C_\beta,\Omega^1)$.
	The commuting involutions
	\[
	\sigma[S:P:W]=[-S:P:W],\qquad
	\tau[S:P:W]=[S:W:P]
	\]
	give genus-one quotients $E_2,E_+,E_-$ for $\sigma,\tau,\sigma\tau$,
respectively:
\begin{equation}\label{xq:E2}
		E_2:\quad Y^2=\beta P^4+2P^2+\beta,\qquad
		\pi_2(S,P)=(P,S^2-2P),
	\end{equation}
	and
	\begin{equation}\label{xq:maps}
		\begin{aligned}
			x_+&=\frac{S}{P+1},\qquad v_+=\frac{P-1}{P+1},\\
			Y_+&=(1-\beta)v_+^2+4x_+^2-(1+3\beta),\\
			x_-&=\frac{S}{P-1},\qquad v_-=\frac{P+1}{P-1},\\
			Y_-&=(1-\beta)v_-^2-4x_-^2-(1+3\beta).
		\end{aligned}
	\end{equation}
	\begin{equation}\label{xq:Epm}
		E_\pm:\quad Y_\pm^2=8\bigl[(1+\beta)x_\pm^4\mp4\beta x_\pm^2
		+\beta(1+\beta)\bigr].
	\end{equation}
	The degree-two maps $\pi_\pm=(x_\pm,Y_\pm)$ and $\pi_2$ extend to the smooth
models. For $\omega_2=dP/Y$, $\omega_\pm=dx_\pm/Y_\pm$,
\begin{equation}\label{xq:pullbacks}
		\pi_2^*\omega_2=\eta_2,\qquad
		\pi_+^*\omega_+=\frac{\eta_0-\eta_1}{4},\qquad
		\pi_-^*\omega_-=\frac{\eta_0+\eta_1}{4}.
	\end{equation}
	Consequently, over $\mathbb C$,
	\begin{equation}\label{xq:isogeny}
		\operatorname{Jac}(\mathcal C_\beta)
		\sim\operatorname{Jac}(E_2)\times
		\operatorname{Jac}(E_+)\times\operatorname{Jac}(E_-)
		\simeq\operatorname{Jac}(E_2)\times\operatorname{Jac}(E_+)^2.
	\end{equation}
	The last isomorphism is of abelian varieties without a specified polarization.
\end{theorem}

\begin{proof}
Let $F$ denote the quartic. A singular point with $S=0$ would give
$\beta=0$ or $W^2=\beta P^2$, $P^2=\beta W^2$ with $PW\ne0$.
For $S\ne0$, $F_S=0$ gives $S^2=2PW$, and the other partials force
$W^2=-\beta P^2$, $P^2=-\beta W^2$. Both alternatives contradict the
hypotheses. Thus the plane quartic is smooth of genus three. Its
adjunction basis is \eqref{xq:differentials}, since $F_S=4SY$ on $W=1$.

The fixed lines $S=0$, $P=W$, $P=-W$ each meet it in four simple points
(the last two give $S^4\mp4S^2P^2+2(1-\beta)P^4=0$). The isolated fixed
points $[1:0:0],[0:1:-1],[0:1:1]$ are off the curve. Riemann--Hurwitz
gives genus one for each quotient. Completing the square in $S^2$ gives
\eqref{xq:E2}. In $x=x_+$, $v=v_+$ the equation is
\[
 8x^4-8x^2(1-v^2)+(1-\beta)(1+v^4)-2(1+3\beta)v^2=0.
\]
Here $\tau$ changes $v$ to $-v$, so $x,v^2$ generate the invariant
field. Completing the square in $v^2$ gives the plus model
\eqref{xq:Epm}; the coordinates with $P-1$ give the minus model.
Differentiating the maps with
$dS/dP=(S^2-P+\beta P^3)/[S(S^2-2P)]$ gives \eqref{xq:pullbacks}.
The pullbacks form a basis, so the induced map to the product of the
quotient Jacobians is an isogeny. The map $(x_-,Y_-)\mapsto
(\mathrm i x_-,Y_-)$ identifies the two last factors over $\mathbb C$.
\end{proof}

\begin{corollary}[Three CM factors for the Cayley kernel]
	\label{xq:CMsplit}
	The quotient invariants are
	\begin{equation}\label{xq:invariants}
		j(E_2)=\frac{64(1+3\beta^2)^3}{\beta^2(1-\beta^2)^2},\qquad
		j(E_\pm)=\frac{64(3\beta+1)^3(\beta+3)^3}
		{(\beta+1)^2(\beta-1)^4}.
	\end{equation}
	In particular, at $(c,b)=(5,-2)$,
	so that $\beta=-1/3$,
	\begin{equation}\label{xq:CMproduct}
		\boxed{\operatorname{Jac}(\mathcal C_{-1/3})
			\sim_{\mathbb C}E_{1728}\times E_0^2.}
	\end{equation}
	Here $E_{1728}$ and $E_0$ denote elliptic curves with the indicated
	$j$-invariants, and every factor has complex multiplication.
\end{corollary}

\begin{proof}
For $y^2=ax^4+bx^2+c$, the binary quartic invariants
$I=b^2+12ac$, $J=72abc-2b^3$ give
$j=6912I^3/(4I^3-J^2)=16(b^2+12ac)^3/[ac(b^2-4ac)^2]$.
Substitute \eqref{xq:E2} and \eqref{xq:Epm}. At $\beta=-1/3$ the three
values are $1728,0,0$, giving \eqref{xq:CMproduct} by \eqref{xq:isogeny}.
\end{proof}

\begin{proposition}[Which factors carry the sampling differential]
	\label{xq:sampling}
	For an admissible kernel satisfying \eqref{geo:nondegenerate},
	the quarter-weight curve is \eqref{xq:curve} with $\beta=b/(c+1)$.
	On the compatible algebraic continuation of the Mellin germ,
	\begin{equation}\label{xq:V4}
		V_4(z):=\frac{I_{1/4}(z)}{\sqrt2\,\pi}
		=\frac1{A S(S^2-2P)},\qquad P^4=\frac CA.
	\end{equation}
	In particular,
	\begin{equation}\label{xq:samplingdecomposition}
		\boxed{\quad A V_4\,dP=\eta_0
			=2(\pi_+^*\omega_++\pi_-^*\omega_-),\qquad
			A S V_4\,dP=\pi_2^*\omega_2.\quad}
	\end{equation}
	At $(c,b)=(5,-2)$, $\eta_0$ lies in the two $j=0$ components and $S\eta_0$
in the $j=1728$ component.
\end{proposition}

\begin{proof}
Take $u^4=-r_+$, $v^4=-r_-$ on the branches of \eqref{eq:sampling-root-form}
and put $S=u+v$, $P=uv$. Then
$u^4-v^4=(u-v)S(S^2-2P)=-\sqrt D/A$ and
$I_{1/4}=-\sqrt2\,\pi(u-v)/\sqrt D$. This gives \eqref{xq:V4} on the
sampling curve identified in Theorem~\ref{geo:degreegenus}.
Adding the two pullbacks in \eqref{xq:pullbacks} gives \eqref{xq:samplingdecomposition}.
\end{proof}

\begin{theorem}[Quotient periods and gamma integrals]
	\label{xq:gammaperiods}
	At $\beta=-1/3$, let $\gamma_2$ be the real oval of $E_2$ above
	$[\sqrt2-1,\sqrt2+1]$, oriented positively for $dP/Y$.
	For each of $E_+,E_-$ let $\gamma_\pm$ be its entire real oval,
	oriented positively for $dx_\pm/Y_\pm$.  Then
	\begin{align}
		\Omega_2:=\int_{\gamma_2}\omega_2
		&=\frac{\sqrt6\,\Gamma(1/4)^2}{4\sqrt\pi},\label{xq:Omega2}\\
		\Omega_+:=\int_{\gamma_+}\omega_+
		&=\frac{3\sqrt3\,\Gamma(1/3)^3}{2^{10/3}\pi},&
		\Omega_-:=\int_{\gamma_-}\omega_-
		&=\frac{3\Gamma(1/3)^3}{2^{10/3}\pi}.\label{xq:Omegapm}
	\end{align}
	Equivalently, the following positive real integrals are
	\begin{align}
		\int_{\sqrt2-1}^{\sqrt2+1}
		\frac{dP}{\sqrt{6P^2-P^4-1}}
		&=\frac{\Gamma(1/4)^2}{4\sqrt{2\pi}},\label{xq:gammaquarter}\\
		\int_{\sqrt{2/\sqrt3-1}}^\infty
		\frac{dx}{\sqrt{3x^4+6x^2-1}}
		&=\frac{\sqrt3\,\Gamma(1/3)^3}{2^{10/3}\pi}.
		\label{xq:gammathird}
	\end{align}
\end{theorem}

\begin{proof}
On $E_2$ we have $3Y^2=6P^2-P^4-1$. For $a=\sqrt2-1$ set
\[
 U=(aP-1)/(P+a),\qquad V=\sqrt6\,aY/(P+a)^2 .
\]
Then $V^2=U^3-U$, $dP/Y=(\sqrt3/2)dU/V$, and the stated oval maps to
$-1\le U\le0$. Substitutions $t=-U$, $r=t^2$ give the latter period
$2\int_{-1}^0dU/\sqrt{U^3-U}=B(1/4,1/2)=\Gamma(1/4)^2/\sqrt{2\pi}$.
This gives \eqref{xq:Omega2}, and division by $2\sqrt3$ gives \eqref{xq:gammaquarter}.

On $E_\pm$ set $y=3Y_\pm/4$, so $y^2=3x^4\pm6x^2-1$ and
$\omega_\pm=(3/4)dx/y$. With $a_\pm=\sqrt{2/\sqrt3\mp1}$ the real maps
\begin{equation}\label{xq:equianharmonicmap}
 U=\frac{2\sqrt3\,a_\pm}{x-a_\pm}+\sqrt3\mp1,\qquad
 V=\frac{\sqrt3\,a_\pm y}{(x-a_\pm)^2}
\end{equation}
are birational onto $V^2=U^3\pm1$ and satisfy $dU/V=-2dx/y$.
Each quartic has two real roots and one real oval through both points
at infinity. Real projective isomorphism and positive orientation thus
give $\Omega_\pm=(3/8)L_\pm$, where the full positive real cubic periods are
\begin{align*}
 L_+&=2\int_{-1}^\infty\frac{dU}{\sqrt{U^3+1}}
   =\tfrac23[B(1/3,1/2)+B(1/3,1/6)]
   =\frac{\sqrt3\,\Gamma(1/3)^3}{2^{1/3}\pi},\\
 L_-&=2\int_1^\infty\frac{dU}{\sqrt{U^3-1}}
   =\tfrac23 B(1/6,1/2)=\frac{\Gamma(1/3)^3}{2^{1/3}\pi}.
\end{align*}
Here the plus integral splits at $0$; gamma reflection and duplication
give the last equalities. By evenness the positive period of $dx/y$ on
$E_+$ is also $4\int_{a_+}^\infty dx/\sqrt{3x^4+6x^2-1}$, proving
\eqref{xq:gammathird}.
\end{proof}

\begin{corollary}[Transferred periods for the Cayley kernel]
	\label{xq:actualsamplingperiods}
	Fix $(c,b)=(5,-2)$, and let
	$\pi_j^!:H_1(E_j,\mathbb Z)\to H_1(\mathcal C_{-1/3},\mathbb Z)$
	denote homological transfer for the degree-two map $\pi_j$.
	For the ovals of Theorem~\ref{xq:gammaperiods}, define the integral
	cycles $\Gamma_j=\pi_j^!\gamma_j$.  Then
	\begin{align}
		\int_{\Gamma_+}\frac{A I_{1/4}(z)}{\sqrt2\,\pi}\,dP
		&=\frac{3\sqrt3\,\Gamma(1/3)^3}{2^{4/3}\pi},\label{xq:sampleplus}\\
		\int_{\Gamma_-}\frac{A I_{1/4}(z)}{\sqrt2\,\pi}\,dP
		&=\frac{3\Gamma(1/3)^3}{2^{4/3}\pi},\label{xq:sampleminus}\\
		\int_{\Gamma_2}\frac{A S I_{1/4}(z)}{\sqrt2\,\pi}\,dP
		&=\frac{\sqrt6\,\Gamma(1/4)^2}{2\sqrt\pi},&
		\int_{\Gamma_2}\frac{A I_{1/4}(z)}{\sqrt2\,\pi}\,dP&=0.
		\label{xq:samplelemniscatic}
	\end{align}
	Here $P^4=(5-z)/(1-5z)$. Continue along the indicated integral cycles,
which need not be primitive.
\end{corollary}

\begin{proof}
Transfer a loop avoiding the branch values by summing its lifts. The
sum is closed, and a small loop around a branch value transfers to a
boundary, so the map descends to integral homology. The transferred
class is deck-invariant and satisfies $\pi_{j*}\pi_j^!=2$.
Now $\sigma^*\eta_0=-\eta_0$, $\sigma^*\eta_2=\eta_2$,
$\tau^*\eta_0=-\eta_1$, $\tau^*\eta_1=-\eta_0$. Hence
$\pi_-^*\omega_-$ is odd under $\tau$, and $\pi_+^*\omega_+$ is odd
under $\sigma\tau$. Such terms integrate to zero on the corresponding
transfer. Formula \eqref{xq:samplingdecomposition} therefore gives
$\int_{\Gamma_\pm}\eta_0=4\Omega_\pm$,
$\int_{\Gamma_2}\eta_2=2\Omega_2$ and $\int_{\Gamma_2}\eta_0=0$.
Theorem~\ref{xq:gammaperiods} supplies their values.
\end{proof}

\section*{Further questions}
The two-step relation of Lemma~\ref{lem:twostep} propagates the formula on
\(a+s=1\) to the lines \(a+s=2j+1\). Is there a comparable evaluation
on the adjacent line \(a+s=2\) for \(a\ne\tfrac12\)?

At \(a=\tfrac12\), folding the third moment of \(\log\tau_r\) at \(t=1\)
and setting \(t=x^2\) gives
\[
 m_3=2\int_0^1\log^3\!\left(\frac{x^2+r}{1+rx^2}\right)\frac{dx}{1-x^2}.
\]
After \(y=1-x\), the logarithm is a sum of logarithms of linear factors,
each normalized to one at \(y=0\). The shuffle identity and one further
integration express \(m_3\) in weight-four multiple polylogarithms
\cite{Panzer}. Is there a short expression in classical polylogarithms?
For the singular factors, can the mixed triple moment of
Corollary~\ref{cor:triple} be extended to arbitrary products by an
analogue of Theorem~\ref{thm:moments}?

\end{document}